\documentclass[11pt,letterpaper]{amsart}
\usepackage[plainpages=false]{hyperref}
\usepackage{amsfonts,latexsym,rawfonts,amsmath,amssymb,amsthm, mathrsfs, lscape, calligra}
\usepackage{verbatim}
\usepackage{enumerate}
\usepackage{centernot}
\usepackage{mathtools}
\usepackage{tikz}
\usepackage{tikz-cd}

\usepackage[all]{xy}
\usepackage{graphicx,psfrag}

\usepackage{array, tabularx, color}

\usepackage{setspace}

\newtheorem{thm}{Theorem}[section]
\newtheorem{cor}[thm]{Corollary}
\newtheorem{lem}[thm]{Lemma}

\newtheorem{prop}[thm]{Proposition}

\theoremstyle{remark}
\newtheorem{rmk}[thm]{Remark}

\theoremstyle{definition}
\newtheorem{defi}[thm]{Definition}

\newcommand{\Ebold}{{\bf E}}

\newcommand{\CBbb}{\mathbb C}

\newcommand{\RBbb}{\mathbb R}

\newcommand{\Ecal}{\mathcal E}
\newcommand{\Fcal}{\mathcal F}
\newcommand{\Gcal}{\mathcal G}
\newcommand{\Hcal}{\mathcal H}

\newcommand{\Ocal}{\mathcal O}

\newcommand{\Tcal}{\mathcal T}

\newcommand{\PU}{\mathsf{PU}}

\DeclareMathOperator{\Hom}{Hom}

\DeclareMathOperator{\id}{id}
\DeclareMathOperator{\ch}{ch}

\DeclareMathOperator{\rank}{rank}

\DeclareMathOperator{\dvol}{dvol}
\DeclareMathOperator{\Sym}{Sym}

\DeclareMathOperator{\tr}{tr}

\DeclareMathOperator{\Img}{Im}
\DeclareMathOperator{\Ker}{Ker}

\DeclareMathOperator{\PD}{PD}

\newcommand{\dbar}{\bar\partial}

\numberwithin{equation}{section}

\makeatletter
\@namedef{subjclassname@2020}{\textup{2020} Mathematics Subject Classification}
\makeatother

\begin{document}

\title[Admissible Hermitian-Yang-Mills connections]{Admissible Hermitian-Yang-Mills connections over normal varieties}
\author[Chen]{Xuemiao Chen}
\address{Department of Pure Mathematics, University of Waterloo, Waterloo, Ontario, Canada N2L 3G1}
\email{x67chen@uwaterloo.ca}


\begin{abstract}
In this paper, we first prove a complete version of the Donaldson-Uhlenbeck-Yau theorem over normal varieties, including normal K\"ahler varieties and projective normal varieties with multiple polarizations. In particular, this gives the polystability of reflexive sheaves under symmetric and exterior powers and tensor products. As a consequence of the singular Donaldson-Uhlenbeck-Yau theorem, the complete Hitchin-Kobayashi correspondence over normal varieties smooth in codimension two is built by showing that an admissible Hermitian-Yang-Mills connection defines a polystable reflexive sheaf. Furthermore, it is shown that the Hermitian-Yang-Mills connection gives a lower bound for the discriminants of any K\"ahler resolutions, which gives a Bogomolov-Gieseker inequality over normal varieties and a characterization of the equality using projectively flat connections. We discuss typical cases including normal surfaces and varieties smooth in codimension two where we could simplify the Bogomolov-Gieseker inequality and endow it with topological meanings. We also prove the Bogomolov-Gieseker inequality for semistable reflexive sheaves and characterize the class of semistable sheaves that satisfy the Bogomolov-Gieseker equality. Finally, we give a new criteria for when a normal K\"ahler variety with trivial first Chern class is a finite quotient of torus. 
\end{abstract}
\maketitle
\tableofcontents

\thispagestyle{empty}
\bibliographystyle{amsplain}
\section{Introduction}
The celebrated Donaldson-Uhlenbeck-Yau theorem (\cite{Donaldson:87a} \cite{UhlenbeckYau:86}) confirms the existence of  Hermitian-Yang-Mills metrics on stable holomorphic vector bundles over compact K\"ahler manifolds. Two important corollaries include the polystability of reflexive sheaves under symmetric and exterior powers and tensor products, and the Bogomolov-Gieseker inequality for stable bundles over K\"ahler manifolds together with a characterization of the equality using projectively flat metrics. There are also further important developments focusing on different aspects:  stable Higgs bundles over K\"ahler manifolds by Simpson (\cite{Simpson:87}); stable bundles over general compact complex manifold with Gauduchon metrics by Li and Yau (\cite{LiYau:87}) and over surface by Buchdahl (\cite{Buchdahl:88}), with a detailed account by Lubke and Teleman (\cite{LubkeTeleman:95}); stable reflexive sheaves over K\"ahler manifolds by Bando and Siu (\cite{BandoSiu:94}); the most general version over Hermitian manifolds by Lubke and Teleman (\cite{TelemanLubke:06}). The other direction of the Hitchin-Kobayashi correspondence is much simpler and first proved by Kobayabshi and Lubke (\cite{Kobayashi:87}). 

The goal of this paper is to prove the singular Donaldson-Uhlenbeck-Yau theorem over normal varieties in complete generality and give various applications in the singular setting. However, the ``easy" direction of the Hitchin-Kobayashi correspondence is much more involved in the singular case in general. As a consequence of the singular Donaldson-Uhlenbeck-Yau theorem, we will build the complete Hitchin-Kobayashi correspondence over normal varieties smooth in codimension two. 

Now we explain our main motivations for this paper. 

First, in recent work \cite{ChenWentworth:21c} where we proved a singular Donaldson-Uhlenbeck-Yau theorem over a class of projective normal varieties, we need to make the assumption that either the projective normal variety has codimension at least three singular set or it comes from the limit of smooth projective varieties with induced K\"ahler metric from an ambient smooth variety. The techniques developed in \cite{ChenWentworth:21c} has its limits when dealing with normal K\"ahler varieties in general.  The first goal of this paper is to remove those assumptions. For this, we will have to deal with more subtle analytic problems and gauge theoretical phenomenon in the critical case due to lack of algebraic geometric tools. There are even simple questions that can be easily solved in the smooth case but not over singular varieties. For example, even in the case of rank $1$ reflexive sheaves, the existence of Hermitian-Yang-Mills metrics is nontrivial due to the fact that we are working with singular varieties. Also, a priori, it is not known that an admissible Hermitian metric on a reflexive sheaf does compute the slope of the sheaf, not even to mention the subsheaf. In particular, the easy direction of Hitchin-Kobayashi correspondence that admitting an admissible Hermitian-Yang-Mills metrics imply the sheaf being polystable is highly nontrivial unlike the smooth case, due to lack of the Chern-Weil formula and a sheaf extension result as \cite{BandoSiu:94} in the singular setting. We could pass by those subtle problems in \cite{ChenWentworth:21c} by working in the projective case with mild singularities and assuming that the metric on the normal variety is induced from some ambient smooth manifolds.  On the other hand, this makes the result not sufficient even when we want to deal some related problems with varieties smooth in codimension two (\cite{CGG:21}). To solve the problem in general, we will need to have new ideas and deal with more subtle phenomenon due to the critical singularity assumption.

The second motivation comes from the searching for a Bogomolov-Gieseker inequality over general normal varieties.
In the general case, the picture for the Bogomolov-Gieseker inequality is, a priori, not clear due to the fact that the singularities of the varieties could contribute. For example, in the normal surface case, it consists of points, thus it is a difficult problem to weigh the contribution of such singularities on both the algebraic and analytic sides. It seems very mysterious about how the algebraic and analytic sides intertwine in general.  It turns out that the correct Bogomolov-Gieseker inequality is surprisingly simple which comes from the basic phenomenon that when we shrink the exceptional divisor on the resolution, there could be loss of Yang-Mills energy. Thus the right picture is given by comparing the quantities on any resolutions and the resulting Hermitian-Yang-Mills connections. This recovers the Bogomolov-Gieseker inequality in the smooth case. At the same time, it is sufficient to give a few interesting and nontrivial applications. 

Another related motivation comes from the recent work \cite{ChenWentworth:21b} and \cite{ChenWentworth:21a} which generalizes the classical Donaldson-Uhlenbeck-Yau theorem to complex manifolds with Hermitian metrics of Hodge-Riemann type. More precisely, it was observed in \cite{ChenWentworth:21b} that a new class of Gauduchon metrics can be given by the so-called balanced metrics of Hodge-Riemann type, of which the multipolarizations provide a natural class of examples.  Namely, given  $(n-1)$ K\"ahler forms $\omega_1\cdots, \omega_{n-1}$,  $\omega_1 \wedge \cdots \omega_{n-1}$ defines a balanced metric. More importantly, the Hodge-Riemann property still holds for $\omega_1\wedge \cdots \omega_{n-2}$ (see \cite{Timorin:98}) which gives the Bogomolov-Gieseker inequality in the multipolarization setting (see Section \ref{DUY for mutipolarizations}). The latter had been known when $\omega_i$ are all Hodge metrics which follows from the restriction theorem (\cite{Langer:04} \cite{HuybrechtsLehn:10}). For a balanced metric coming from the multipolarization above, $[\omega_1 \wedge \cdots \omega_{n-1}]$ lies in the interior of the so-called cone of movable curves of compact complex manifolds (\cite{BDPP:13}). Also, the notion of stable sheaves defined via mutipolarizations (\cite{GrebToma:17}) already plays an important role in compactifying the moduli space of semistable sheaves over projective manifolds in higher dimensions. More generally the slope stability via movable class has been defined and studied on normal varieties, and it is a very important and useful concept in birational geometry (\cite{GKP:16b} \cite{CP:19}).

Given the discussions above, we will study the gauge theoretical side for stable reflexive sheaves over normal varieties endowed with multiple K\"ahler metrics. The key difficulty in this process is to deal with the information near the  singularities of the base variety, together with the new features by involving multiple K\"ahler metrics. 

\subsection{Singular Donaldson-Uhlenbeck-Yau theorem} Suppose $X$ is a normal variety of dimension $n$ endowed with $(n-1)$ K\"ahler forms $\omega_1,\cdots, \omega_{n-1}$ (see Section \ref{Kaehler forms} for definitions).  Through the equation $\omega^{n-1}=\omega_1 \wedge \cdots \omega_{n-1}$, this defines a balanced metric $\omega$. A stable reflexive sheaf can be defined by passing to a particular resolution $p: \hat X \rightarrow X$ so that $\hat{\Ecal}:=(p^*\Ecal)^{**}
$ is stable with respect to $p^*\omega_1 \wedge \cdots p^*\omega_{n-1}.$ 
The stability is independent of the resolutions since everything is smooth in codimension one (see Section \ref{Chern classes}). 

\begin{thm}\label{Singular DUY theorem}
Suppose $\Ecal$ is a stable reflexive sheaf over a normal variety  $(X, \omega_1 \wedge \cdots \omega_{n-1})$. Then there exists an admissible Hermitian-Yang-Mills metric on $\Ecal$.  Moreover, such a metric is unique up to scaling.
\end{thm}
\begin{rmk}
In the K\"ahler variety case, i.e. $\omega=\omega_1=\cdots \omega_{n-1}$, there has been some recent progress on the singular Donaldson-Uhlenbeck-Yau theorem. In the projective case (\cite{ChenWentworth:21c}), if $\omega$ is the restriction of a Hodge metric from the ambient smooth variety, this has been proved by assuming the base is smooth in codimension two; a singular Donaldson-Uhlenbeck-Yau theorem has also been shown for stable sheaves over projective normal varieties which come from limits of stable sheaves over smooth projective varieties. Assuming a uniform Sobolev constant control for the resolution with perturbed K\"ahler metrics, the argument in \cite{ChenWentworth:21c} can be used to prove slightly more general results for normal projective varieties and subvarieties of K\"ahler manifolds smooth in codimension two. However, the control of the Sobolev constants for the perturbed K\"ahler metrics does not seem to follow from any known results unlike the smooth case in \cite{BandoSiu:94}. On the other hand, to use the heat flow to prove the singular Donaldson-Uhlenbeck-Yau theorem, the control of the Sobolev constants is very crucial. Assuming the existence of the uniform Sobolev constant control, the singular Donaldson-Uhlenbeck-Yau theorem could be obtained over normal K\"ahler varieties using the heat flow method (see \cite{Ou:22}). 
\end{rmk}
As a direct corollary, this gives
\begin{cor}\label{Tensor being polystable}
Suppose $\Ecal_1, \Ecal_2$  are stable reflexive sheaves over $(X, \omega_1 \wedge \cdots \omega_{n-1})$, then $(\Sym^{k}\Ecal_1)^{**}, (\wedge^k \Ecal_1)^{**}$ and $(\Ecal_1\otimes\Ecal_2)^{**}$ are all polystable. 
\end{cor}

\subsection{Hitchin-Kobayashi correspondence over normal varieties smooth in codimension two} By working over normal varieties smooth in codimension two, the complete Hitchin-Kobayashi correspondence could be built by using the singular Donaldson-Uhlenbeck-Yau theorem together with Siu's sheaf extension result. For this, we have

\begin{cor}\label{Cor1.12}
An admissible Hermitian-Yang-Mills connection over a normal variety $(X, \omega_1\wedge \cdots \omega_{n-1})$ smooth in codimension two defines a polystable reflexive sheaf. 
\end{cor}

\begin{proof}
By \cite{Siu:71}, we know an admissible Hermitian-Yang-Mills connection defines a unique coherent reflexive sheaf since $X$ is smooth in codimension two. Now the polystability follows from  Corollary \ref{Cor3.5} which is essentially a consequence of the singular Donaldson-Uhlenbeck-Yau theorem obtained above. 
\end{proof}

In particular, combined with the singular Donaldson-Uhlenbeck-Yau theorem above, we have 
\begin{thm}
Over a normal variety $X$ smooth in codimension two endowed with $n-1$ K\"ahler forms $\omega_1,\cdots \omega_{n-1}$, there exists a one-to-one correspondence between the space of isomorphism classes of stable reflexive sheaves and the space of gauge equivalent classes of admissible Hermitian-Yang-Mills connections over $(X, \omega_1\wedge \cdots \omega_{n-1})$.  
\end{thm}

\begin{rmk}\label{Remark 2.10}
\begin{itemize}
\item For the full Hitchin-Kobayashi correspondence over general normal varieties,  given the singular Donaldson-Uhlenbeck-Yau theorem above, the essential difficulty lies in proving an admissible Hermitian-Yang-Mills connection defines a coherent reflexive sheaf. More precisely, one needs to construct enough sections of the sheaf across the singular set. 

\item This statement is only on the set level. It seems to the author that an improved map such as an \emph{analytic} isomorphism as the smooth case is very subtle near singularities of the base. For example, near the singularities, it involves subtle analysis to make sense of a reasonably good moduli space of the analytic solutions. In the normal surface case, as pointed out by Donaldson, this could be related to \cite{Donaldson:02}. We leave this for future work. 

\end{itemize}
\end{rmk}

\subsection{Bogomolov-Gieseker inequality} Now we explain how to build a Bogomolov-Gieseker inequality in the general case. As already mentioned above, this is essentially a consequence of the basic phenomenon that there could be loss of Yang-Mills energy when we shrink the exceptional divisor in the proof of the Donaldson-Uhlenbeck-Yau theorem above.

By the Hodge-Riemann property for multipolarizations, the Hermitian-Yang-Mills metric $H$ obtained above gives an analytic Bogomolov-Gieseker inequality. Unless the variety is smooth in codimension two, it is in general expected that it does not compute any corresponding algebraic quantities. However, motivated by the gauge theoretical picture in the proof of the singular Donaldson-Uhlenbeck-Yau theorem, for the algebraic side, fix any $1\leq i_1< \cdots i_{n-2}\leq n-2$, we define the discriminant as 
$$
\begin{aligned}
&\Delta(\Ecal)[\omega_{i_1}] \cdots  [\omega_{i_{n-2}}]\\
=&\inf_{\text{ K\"ahler } p}\inf_{\hat{\Ecal}\in\Ebold_p } (2rc_2(\hat\Ecal)-(r-1)c_1(\hat\Ecal)^2).[p^*\omega_{i_1}] \cdots  [p^*\omega_{i_{n-2}}]
\end{aligned}
$$
which is intrinsically associated to $\Ecal$, here $p$ is among all the K\"ahler resolutions, i.e. $\hat{X}$ admits a K\"ahler metric and $\Ebold_p$ denotes the space of reflexive sheaves over $\hat{X}$ which are isomorphic to $\Ecal$ away from the exceptional divisor. It turns out the Hermitian-Yang-Mills metric does give a lower bound for the discriminant, thus a version of Bogomolov-Gieseker inequality over normal varieties without requiring the variety being smooth in codimension two. This recovers the classical Bogomolov-Gieseker inequality in the smooth case.

\begin{cor}\label{Bogomolov-Gieseker inequality}
Suppose $\Ecal$ is a stable reflexive sheaf over $(X, \omega_1 \wedge \cdots \omega_{n-1})$ and let $H$ be the admissible Hermitian-Yang-Mills metric as above. Then 
$$
\begin{aligned}
& \Delta(\Ecal)[\omega_{i_1}] \cdots  [\omega_{i_{n-2}}]\\
\geq &\int_{X} (2rc_2(H)-(r-1)c_1(H)^2)\wedge \omega_{i_1}\wedge \cdots \omega_{i_{n-2}} 
\end{aligned}
$$
for any $1\leq i_1<\cdots i_{n-2} \leq n-1$. In particular, 
$$
 \Delta(\Ecal)[\omega_{i_1}] \cdots  [\omega_{i_{n-2}}] \geq 0
$$
where $\Ecal$ is projectively flat if the equality holds, i.e. $\Ecal|_{X^{reg}}$ is defined by a representation $\rho: \pi_1(X^{reg}) \rightarrow \PU(m)$ where $m=\rank \Ecal$.  
\end{cor}

\begin{rmk} 
\begin{itemize}
\item From our perspective, the most important aspect of this theorem lies in that if the stable reflexive sheaf can be resolved in a nice way, then one can conclude good properties about the original sheaf. Such methods are very often used when one studies stable reflexive sheaves in birational geometry or decomposition of singular spaces. To name a few, for example, \cite{CGG:21} \cite{CGGN:22} \cite{CP:19} \cite{GKP:16a} \cite{GKP:16b}. See also Corollary \ref{Theorem 1.18} for such a nontrivial application. 

\item We emphasize that we would not expect the equality in the first inequality to hold in general due to the natural gauge theoretical picture. The significance shows up when we study the moduli space of admissible Hermitian-Yang-Mills connections over normal varieties, this will be an important ingredient since it gives us natural bound on the $L^2$ norm of the curvature. Especially in the normal surface case, this could be related to Donaldson theory (see Section \ref{Over normal surfaces}).  

\item When the base is smooth in codimension two, the Bogomolov-Gieseker inequality is known in the case of projective normal varieties, or more generally K\"ahler normal varieties, and the equality has been characterized with various characteristic classes vanishing conditions (\cite{Wu:21, CGG:21, HoringPeternell:19, GKP:16b}). The novelty of our theorem is that it builds the general version over normal varieties without being smooth in codimension three assumption and gives a characterization of the equality in complete generality.

\item We also refer the readers to the later recent preprints \cite{Ou:24, GP:24} for related results.
\end{itemize}
\end{rmk}

\subsection{Simplifications of Bogomolov-Gieseker inequalities}
Now we discuss various typical and important cases where for the Bogomolov-Gieseker inequality above, we could simplify or give the formula topological meanings. 

\subsubsection{Over normal surfaces}\label{Over normal surfaces}
We first look at the Bogomolov-Gieseker inequality in the normal surface case which could have potential use for understanding Donaldson theory over normal surfaces.

First, when $\Ecal$ is locally free over a normal K\"ahler surface $(X, \omega)$, the discriminant can be computed by the natural pull-back for any resolutions. In this case, since $\Ecal$ is locally free, it defines a bundle over $X$, thus there already exist topological quantities corresponding to the ones defined here. 
 
\begin{cor}[When $\Ecal$ is locally free over a normal surface]\label{Cor1.6}
Suppose $\Ecal$ is a stable locally free sheaf over a normal K\"ahler surface $(X, \omega)$ and let $H$ be the admissible Hermitian-Yang-Mills metric as above. Then 
$$
\begin{aligned}
\Delta(\Ecal)
=& (2rc_2(\Ecal)-(r-1)c_1(\Ecal)^2)\cap X\\
\geq & \int_{X} (2rc_2(H)-(r-1)c_1(H)^2)\\
\geq& 0
\end{aligned}
$$
where $\Ecal$ is projectively flat if the last equality holds, i.e. $\Ecal|_{X^{reg}}$ is defined by a representation $\rho: \pi_1(X^{reg}) \rightarrow \PU(m)$ where $m=\rank \Ecal$.
\end{cor}

\begin{rmk}
\begin{itemize}
\item An interesting aspect about the first equality lies in that for any resolution $p$, $p^*\Ecal$ will compute the smallest one among the extensions $\Ebold_{p}$ which is related to a lemma of Du Val (\cite[Proposition 2.1.12]{Nemethi:22}). It states the natural intersection matrix given by intersecting irreducible components of the exceptional divisor is negative definite. 

\item Over singular varieties, existence of locally free sheaves is in general a very nontrivial question, but in the normal surface case, abundant locally free sheaves have been constructed in \cite{SchroerVezzosi04}. Combined with the results above, this could be a starting point to study Donaldson theory over normal surfaces, and we leave this for future work. 

\end{itemize}
\end{rmk}

The second simplification is that the quantity can be computed by using only the minimal resolution for a normal surface 
 
\begin{cor}[On the minimal resolution of a normal surface]\label{Minimal resolution}
Suppose $\Ecal$ is a stable reflexive sheaf over a normal K\"ahler surface $(X, \omega)$ and let $H$ be the admissible Hermitian-Yang-Mills metric as above. Then 
$$
\begin{aligned}
\Delta(\Ecal)=& \inf_{\hat{\Ecal}\in\Ebold_{p^{min}} } (2rc_2(\hat\Ecal)-(r-1)c_1(\hat\Ecal)^2) \\
 \geq &\int_{X} (2rc_2(H)-(r-1)c_1(H)^2)
\\ 
\geq& 0
\end{aligned}
$$
for any $1\leq i_1<\cdots i_{n-2} \leq n-1$, where $\Ecal$ is projectively flat if the last equality hold, i.e. $\Ecal|_{X^{reg}}$ is defined by a representation $\rho: \pi_1(X^{reg}) \rightarrow \PU(m)$ where $m=\rank \Ecal$.
\end{cor}

\begin{rmk}
It remains as an interesting question in general to characterize when the Chern-Weil formula holds. We expect the right picture for this has connections with \cite{Langer:00}. We leave this for future work. 
\end{rmk}

\subsubsection{Smooth in codimension two}
When $X$ is a normal variety with multiple K\"ahler metrics smooth in codimension two, the quantity can be computed by using any resolutions. Also a Chern-Weil formula holds when the base has isolated singularities.

\begin{cor}[When $X$ is smooth in codimension two]\label{Cor1.10}
Suppose $\Ecal$ is a stable reflexive sheaf over $(X, \omega_1 \wedge \cdots \omega_{n-1})$ smooth in codimension two and let $H$ be the admissible Hermitian-Yang-Mills metric as above. 
\begin{enumerate}
\item The following holds
$$
\begin{aligned}
&\Delta(\Ecal)[\omega_{i_1}] \cdots  [\omega_{i_{n-2}}]\\
=&(2rc_2(\hat\Ecal)-(r-1)c_1(\hat\Ecal)^2).[p^*\omega_{i_1}] \cdots  [p^*\omega_{i_{n-2}}]\\
\geq & \int_{X} (2rc_2(H)-(r-1)c_1(H)^2)\wedge \omega_{i_1}\wedge \cdots \omega_{i_{n-2}}
\end{aligned}
$$
for any resolution $p$ and $\hat{\Ecal}\in \Ebold_p$ and for any $1\leq i_1<\cdots i_{n-2} \leq n-1$, where $\Ecal|_{X^{reg}}$ is projectively flat if the last equality holds.. In particular, when $[\omega_1], \cdots [\omega_{n-1}]$ can be represented by very ample divisors, then  
$$
\begin{aligned}
&\Delta(\Ecal)[\omega_{i_{n-1}}] \cdots  [\omega_{i_{n-2}}]\\
=&\Delta(\Ecal|_{D_{i_1}\cap \cdots D_{i_{n-2}}}) \\
\geq &\int_{X} (2rc_2(H)-(r-1)c_1(H)^2)\wedge \omega_{i_1}\wedge \cdots \omega_{i_{n-2}}\\
\geq &0
\end{aligned}
$$
where $[D_{i_j}]=[\omega_{i_j}]$ so that $D_{i_1}\cap \cdots D_{i_{n-2}}$ is a smooth surface in $X$ and $\Ecal|_{X^{reg}}$ is projectively flat if the last equality holds, i.e. $\Ecal|_{X^{reg}}$ is defined by a representation $\rho: \pi_1(X^{reg}) \rightarrow \PU(m)$ where $m=\rank \Ecal$.

\item  Assume further $X$ has isolated singularities, then 
$$
\begin{aligned}
&\Delta(\Ecal)[\omega_{i_{n-1}}] \cdots  [\omega_{i_{n-2}}]\\
= &\int_{X} (2rc_2(H)-(r-1)c_1(H)^2)\wedge \omega_{i_1}\wedge \cdots \omega_{i_{n-2}}\\
\geq & 0
\end{aligned}
$$
where $\Ecal|_{X^{reg}}$ is projectively flat if the last equality holds, i.e. $\Ecal|_{X^{reg}}$ is defined by a representation $\rho: \pi_1(X^{reg}) \rightarrow \PU(m)$ where $m=\rank \Ecal$.
\end{enumerate}
\end{cor}

\subsection{Bogomolov-Gieseker inequality for semistable sheaves} Now we study the Bogomolov-Gieseker inequality for semistable reflexive sheaves. For this, we suppose $(X, \omega_1\wedge \cdots \omega_{n-1})$ is a normal variety with $n-1$ K\"ahler metrics. Then we have

\begin{thm}\label{Theorem1.16}
Given a semistable reflexive sheaf $\Fcal$ over $(X, \omega_{1} \wedge \cdots \omega_{n-1})$, the following holds
$$\Delta(\Ecal).[\omega_{i_1}] \cdots [\omega_{i_{n-2}}]\geq 0.$$
Suppose for some $\hat{\Ecal}\in \Ebold_p$ and some resolution $p: \hat{X} \rightarrow X$,
$$\Delta(\hat{\Ecal}).[p^*\omega_{i_1}] \cdots [p^*\omega_{i_{n-2}}]=0,
$$
then $\Fcal$ admits a filtration 
$$
0\subset \Fcal_1 \subset \cdots \Fcal_{m}=\Fcal
$$
so that $\Fcal_{i}/\Fcal_{i-1}$ is torsion free and $(\Fcal_{i}/\Fcal_{i-1})|_{X^{reg}}$ is projectively flat, i.e.  $(\Fcal_i/\Fcal_{i-1})|_{X^{reg}}$ is defined by a representation $\rho: \pi_1(X^{reg}) \rightarrow \PU(m_i)$ where $m_i=\rank(\Fcal_i/\Fcal_{i-1})$.
\end{thm}
\begin{rmk}
\begin{itemize}
\item Given the singular Donaldson-Uhlenbeck-Yau theorem we proved above, the key ingredient in this theorem lies in the fact that a weak Hodge-Riemann property still holds on the resolution for multiple semi-K\"ahler classes. This follows essentially from \cite{DinhNguyen:06} which generalizes the classical Hodge-Riemann bilinear relation for a fixed K\"ahler classes to the case of multiple K\"ahler classes.

\item Assume $X$ is smooth in codimension two, by Corollary \ref{Cor1.10}, the condition does not depend on the resolution and thus one even has abundant natural examples from algebraic geometry. As we already see in the projective case with one polarization, this result is already very useful in improving the known algebraic geometric result (see \cite{ChenWentworth:21c}).
\end{itemize}

\end{rmk}

\subsection{Quotients of torus} Now we apply our results to give a new criteria about when a normal complex space with klt singularities is a quotient of a complex torus by a finite group by generalizing the results in \cite{CGG:21}. We refer \cite{CGG:21} for a detailed account on the importance of such problems.

\begin{cor}\label{Theorem 1.18}
Let $(X,[\omega])$ be a compact normal K\"ahler variety with klt singularities satisfying $c_1(X)=0\in H^2(X, \RBbb)$. Suppose there exists a reflexive sheaf $\hat{\Ecal}$ for some resolution $p: \hat{X}\rightarrow X$ with $\hat{\Ecal}|_{X^{reg}}\cong \Tcal_X|_{X^{reg}}$
$$
\Delta(\hat{\Ecal}) [p^*\omega]^{n-2}=0
$$
then $X$ is a quotient of a complex torus by a finite group acting freely in codimension one.
\end{cor}

\begin{proof}
This follows from Corollary \ref{Bogomolov-Gieseker inequality} together with the remark in \cite[Page 3]{CGG:21}. More precisely, by \cite[Theorem A]{Guenancia:15}, we know that the tangent sheaf $\Tcal_{X}$ is polystable with slope being zero. By Corollary \ref{Bogomolov-Gieseker inequality}, we know $\Tcal_X|_{X^{reg}}$ is flat since $c_1(X)=0$. Now the conclusion follows from \cite[Theorem D]{CGGN:22}.
\end{proof}

\begin{rmk}
This generalizes the key direction in Theorem $A$ in \cite{CGG:21} by removing the assumption of $X$ being smooth in codimension two. In \cite{CGG:21}, $X$ is assumed smooth in codimension two, thus $\hat{\Ecal}$ can be chosen to be $\Tcal_{\hat{X}}$ and the condition is equivalent to
$$
\int_{\hat{X}} c_2(\hat{X})\wedge [p^*\omega]^{n-2}=0.
$$
\end{rmk}

\subsection{Sketch of the proof} 
\subsubsection{Regarding singular Donaldson-Uhlenbeck-Yau theorem} The strategy for the proof is standard and has been used in \cite{ChenWentworth:21c} in the projective case. It is done by passing to a resolution of singularities and studying the corresponding gauge theoretical limits, which goes back to \cite{Donaldson:85, UhlenbeckYau:86, BandoSiu:94}. The subtlety lies really in how to take care of various technical difficulties in this process with new ideas and prove it in the multi-K\"ahler setting. Start with any resolution $p:\hat X \rightarrow X$ so that $\hat \Ecal :=(p^*\Ecal)^{**}$ is locally free. Write $X=X^s\cup X^{reg}$ as a union of singular and smooth parts. Fix any K\"ahler metric $\theta$ on $\hat X$. Then one can show $\hat \Ecal$ is stable with respect to $(p^*\omega_1+i^{-1} \theta) \wedge \cdots (p^*\omega_{n-1}+i^{-1} \theta)$ for any $i>>1$, which essentially follows from the boundedness results in \cite{Toma:19}. Thus the Donaldson-Uhlenbeck-Yau theorem for multipolarizations (\cite{ChenWentworth:21b}) gives a family of Hermitian-Yang-Mills metrics $H_i$ with respect to the perturbed metrics. Let $\hat E$ be the underlying smooth bundle for $\hat{\Ecal}$. By normalizing gauge, we get a sequence of Hermitian-Yang-Mills connections on $\hat E$ with a unitary metric $H$. By known gauge theoretical results (\cite{ChenWentworth:21a}), after passing to a subsequence, up to gauge transforms, $A_i$ converges to a limiting Hermitian-Yang-Mills connection $A_\infty$ over $X\setminus Z$ where $Z=X^s\cup \Sigma$ and $\Sigma$ is a codimension two subvariety of $X^{reg}$. The proof is done in the following steps 
\begin{enumerate}
\item there exists a nontrivial map $\Phi: \Ecal_\infty \rightarrow \Ecal$ where $\Ecal_\infty$ the reflexive sheaf defined by $A_\infty$ over $X\setminus X^s$; for any global section $s$ of $\Ecal_\infty$, $\log^+|s|^2 \in  W^{1,2}\cap L^\infty$. The non-triviality of this step essentially lies in that we need to take limits of holomorphic sections over noncompact manifolds. 

\item in the rank $1$ case, show $A_\infty$ computes $\mu(\Ecal)$ and thus conclude   $\Phi$ is an isomorphism. Note there is no Chern-Weil formula in the singular setting and the existence of admissible Hermitian-Yang-Mills connections in the rank one case is already highly nontrivial.

\item prove $\Phi$ is an isomorphism for general rank. This crucially depends on (1) and (2). 
\end{enumerate}

For (1), we crucially use the fact that $Z$ has codimension two as in \cite{ChenSun:18}. Fix a smooth metric $H'$ on $\Ecal$. The idea is to take limit of the sections given by the identity map $\id \in \Hom(\Ecal, \Ecal)$.  We can take a precompact exhaustion $X^\epsilon$ of $X\setminus Z$ where $X\setminus X^\epsilon \rightarrow Z$ as $\epsilon\rightarrow 0$. Normalize the identity map to have $L^2$ norm equal to one over a fixed region $X^\epsilon$ with respect to the metric $H_i^*\otimes H'$ on $\Hom(\Ecal, \Ecal)$ and the metric on $X$ given by $\omega_1\wedge \cdots 
\omega_{n-1}$. This gives us a sequence of holomorphic sections. By standard elliptic theory, passing to a subsequence, one can take a limit of this sequence. But the problem is that we are working with noncompact base, thus the limit might be trivial. However, we can prove a useful property using the structure of $Z$: for any $z\in X^{\frac{\epsilon}{2}}$, there exists a holomorphic curve $D\subset X^{reg}$ passing $z$ and $\partial D \subset X^{\epsilon}$. This enables us to restrict our sections to $D$, and apply maximum principle to get control over $X^{\frac{\epsilon}{2}}$. In particular, we get a nontrivial limit over $X^\epsilon$ which by induction implies the existence of a nontrivial limit over $X^{reg}.$ For the regularity statement about sections $s$ of $\Ecal_\infty$, one can first show $\log^{+}|s|^2\in W^{1,2}_{loc}$ by an adaption of Bando and Siu's argument (\cite[Theorem 2]{BandoSiu:94}) using one dimensional slices instead of two dimensional slices. Then it follows from \cite{LiTian1995} and \cite{Simon:83} that there exists a global Sobolev inequality for $W^{1,2}(X^{reg})$ functions, thus one can apply Moser iteration to get the $L^\infty$ bound (see Section \ref{Regularity results}). 

For (2), if we assume $\rank \Ecal=1$, then the Remmert-Stein extension theorem implies $\Ecal_\infty$ can be extended to be a reflexive sheaf over $X$ since we have a nontrivial map $\Phi:\Ecal_\infty \rightarrow \Ecal$. If we can show $\mu(\Ecal_\infty)=\mu(\Ecal)$ which will be a very crucial fact needed in (3) as well, the map has to be an isomorphism. This relies on a key observation that such a Chern-Weil formula still exists by using the fact that $A_\infty$ comes from the limit of smooth ones on the resolutions and the fact that $X^s\subset X$ is a complex subvariety of codimension at least two. 

For (3), by Siu's theorem (\cite{Siu:75}), one can show actually the saturation $\Gcal$ of the image of $\Phi$ defines a coherent analytic subsheaf of $\Ecal$. If $\rank \Gcal<\rank\Ecal$, then $\mu(\Gcal)<\mu(\Ecal)$ since $\Ecal$ is stable. By applying the Weizenb\"ock formula to sections of $  \Hom(\wedge^{\rank \Gcal} \Ecal_\infty, \det \Gcal)$, this will give a contradiction. Here we need the crucial fact that $\det \Gcal$ admits a Hermitian-Yang-Mills metric which does compute the slope of $\det \Gcal$ by step (2). In particular, $\Phi$ has full rank. Now by (2), we know $\det(\Ecal_\infty)\cong\det(\Ecal)$ which will force $\Phi$ to be an isomorphism. The conclusion follows.

\subsubsection{Bogomolov-Gieseker inequality} As we mentioned in the introduction, the Bogomolov-Gieseker inequality follows from the gauge theoretical picture when we shrink the exceptional divisor naturally. We first deal with $p:\hat{X} \rightarrow X$ so that $\hat{\Ecal}=(p^*\Ecal)^{**}$ is locally free. On the resolution, the quantities in the Bogomolov-Gieseker inequality can be directly computed by the perturbed Hermitian-Yang-Mills metrics. Now the limit of this equation will give us what we need. This is due to the fact that by the Hodge-Riemann property for multipolarizations, the integrand given by the Hermitian-Yang-Mills metrics on the resolution defines a sequence of Radon measures with uniformly bounded mass on $\hat X$, thus the inequality follows from Fatou's lemma. For general resolutions $p:\hat{X} \rightarrow X$, this follows from the same argument by our main theorem applied to stable reflexive sheaves over $\hat X$ and that the Chern-Weil formula still holds for such with admissible Hermitian-Yang-Mills metrics.  

Now the various cases of simplifications follow from some topological arguments as well as some more subtle analytic properties.  

The Bogomolov-Gieseker inequality for semistable sheaves follows essential from results on Hodge-Riemann properties for multiple K\"ahler metrics \cite{DinhNguyen:06} together with our main results for stable reflexive sheaves.

\subsection*{Acknowledgment}
The author would like to thank Richard Wentworth for helpful discussions. The author would also like to thank the referees for pointing out a mistake, valuable suggestions and questions which improved the presentations of the paper, and pointing out related references.  This work is partially supported by NSERC and the ECR supplement.

\section{Preliminary results}
\subsection{Varieties with mutiple K\"ahler metrics}\label{Kaehler forms}
In this section, we will recall the notion of K\"ahler metrics on normal varieties from \cite{varouchas1989kahler}. 

Let $X$ be a normal variety. We will always write $X=X^s\cup X^{reg}$ where $X^s$ denotes the singular part of $X$ having codimension at least two; $X^{reg}$ denotes the smooth part of $X$. A local function $f$ on $X$ is called to be smooth if for any local embedding of $X$, i.e. $X\cap U \rightarrow U \subset \CBbb^N$, it can be extended to be a smooth function. Now a local smooth strongly plurisubharmonic  function on $X$ means a smooth function which can be extended to be a smooth strongly plurisubharmonic function for some embedding.

\begin{defi}[Definition-Lemma \cite{varouchas1989kahler}]
A K\"ahler metric $\omega$ on $X$ is defined by a cover $\{(U_i, \rho_i)\}$, where $U_i\subset \CBbb^N$ is open subset, $\rho_i$ is a smooth strongly plurisubharmonic function on $\overline{U_i}$, i.e. in a neighborhood of $\overline{U_i}$ in $\CBbb^N$, and $\omega|_{X^{reg}\cap U_i}=\sqrt{-1}\partial\bar{\partial}\rho_i|_{X^{reg}\cap U_i}$. By \cite[Page 23]{varouchas1989kahler}, this gives a C\v ech class $[\omega]\in H^2(X, \mathbb{R})$, which is usually referred as a K\"ahler class on $X$. 
\end{defi}

\begin{rmk}
We will also need to pull back a K\"ahler class through resolutions $p: \hat{X} \rightarrow X$. For this, we will keep using the fact that $p^*[\omega]$ is the same thing as locally pulling back the smooth defining functions $\rho_i$ as $\rho_i\circ p$ in the definition above, i.e. $p^*\omega$ as a smooth de-Rahm class is locally given by $\sqrt{-1}\partial \dbar (\rho_i \circ p)$ (see \cite[Proposition 3.5]{GK:20}). Abusing notation, we will not make a difference between the C\v ech class and the corresponding de Rham class while the context should make it clear. 
\end{rmk}

Now take any smooth resolution of $p: \hat{X} \rightarrow X$. As a consequence of the definition, we have the following properties for K\"ahler metrics on varieties. 
\begin{enumerate}
\item For each $i$, $p^*\omega_i$ defines a smooth de-Rham class in $\hat{X}$ which follows from that we can pull back the smooth plurisubharmonic functions to define the pull-back of the corresponding K\"ahler class; 
\item For dimensional reasons, the C\v ech class $p^*[\omega_1]\wedge \cdots p^*[\omega_{n-1}]|_{p^{-1}(Z)}=0$ if $Z$ has codimension at least two;
\item Similarly $p^*[\omega_{i_1}]\wedge \cdots p^*[\omega_{i_{n-2}}]|_{p^{-1}(Z)}=0$ if $Z$ has codimension at least three.
\end{enumerate}

We need the following well-known fact, for which we include a short explanation 
\begin{lem}
Let $Z\subset X$ be a subvariety containing $X^s$. Suppose $a\in H^*(\hat X, \mathbb{R})$ with $a|_{p^{-1}(Z)}=0$.  Denote by $[\Omega]$  the corresponding de-Rham class for $a$, then
$$
\Omega=\Omega^0+d\Phi
$$
where $\Omega^0$ is a smooth closed form with compact support in $\hat X\setminus p^{-1}(Z).$
\end{lem}
\begin{proof}
Indeed, choose an open neighborhood $U$ of $p^{-1}(Z)$ which deformation retracts onto $p^{-1}(Z)$, then $H^*(U, \mathbb{R})\cong H^*(p^{-1}(Z), \mathbb{R})$ through the restriction map, which can be seen by using the singular Cohomology. Here we used the functorial properties of the isomorphism between the C\v ech cohomology and singular Cohomology over $\RBbb$ \footnote{Later in various places, we will use such a property without mentioning.}. In particular, we know as an element in the C\v ech cohomology, $a|_U$ is also trivial, thus the corresponding de-Rham cohomology class $[\Omega|_U]$ is also trivial, i.e. $\Omega|_U=d\Phi'$ over $U$. Choose a cut-off function $\rho$ which is $1$ near $p^{-1}(Z)$. Then  $\Phi=\rho \Phi'$ does the job. 
\end{proof}

Given this, if $Z$ has codimension at least two, we can always write 
\begin{equation}\label{n-1 forms}
p^*\omega_1  \wedge \cdots p^*\omega_{n-1}=\Omega^0_{n-1}+d\Phi_{n-1}
\end{equation} 
where $\Omega^0_{n-1}$ is compact supported in $\hat{X} \setminus p^{-1}(Z)$ and if $Z$ has codimension at least three then
\begin{equation}\label{n-2 forms}
p^*\omega_{i_1} \wedge \cdots p^*\omega_{i_{n-2}}=\Omega^0_{n-2}+d\Phi_{n-2}
\end{equation}
where $\Omega^0_{n-2}$ is compact supported in $\hat{X}\setminus p^{-1}(Z)$. 

We will also need the following observation
\begin{prop}\label{Kahler metric equivalence}
Given two K\"ahler metrics $\omega_1, \omega_2$, there exists a constant $C>0$ so that 
$$
C^{-1} \omega_1 \leq \omega_2 \leq C \omega_1.
$$
\end{prop}

\begin{proof}
It suffices to build such a bound near each point $z\in X^s$. Cover $X$ with open sets $(U_i, \rho_i)$ where $U_i \subset \CBbb^N$ and $\rho_i$ is a smooth strongly PSH function on $\overline{U_i}$ and $\omega_1=i\partial\bar{\partial} \rho_i$. Assume $z\in U_i$ for some fixed $i$. By assumption, we can always extend the local defining function $\rho_i'$ for $\omega_2$ to be a smooth function $U_i$.  Thus we know 
$$
\sqrt{-1}\partial\bar{\partial} \rho_i'\leq  C_i \sqrt{-1}\partial\bar{\partial} \rho_i.
$$ 
By compactness, we can cover $X$ with finitely many such open sets and choose $C$ to be the largest $C_i$. This gives $\omega_2 \leq C \omega_1$. The other inequality follows from symmetry where we might need to change $C$ a little.
\end{proof}

\begin{cor}\label{Metric Equivalence}
Given $(n-1)$ K\"ahler metrics $\omega_{i_1}, \cdots \omega_{i_{n-1}}$ on $X$, then for some constant $C>0$
$$
C^{-1} \omega_1^{n-1} \leq \omega_1 \wedge \cdots \omega_{n-1} \leq C \omega_1^{n-1}
$$
i.e. for any $(1,0)$ form $\theta$, 
$$C^{-1} \sqrt{-1} \theta \wedge \overline{\theta} \wedge \omega_1^{n-1} \leq \sqrt{-1} \theta \wedge \overline{\theta} \wedge \omega_1 \wedge \cdots \omega_{n-1} \leq C \sqrt{-1} \theta \wedge \overline{\theta} \wedge\omega_1^{n-1}$$
\end{cor}

\subsection{Donaldson-Uhlenbeck-Yau theorem for balanced metrics of Hodge-Riemann type}\label{DUY for mutipolarizations}
In this section, we recall some results from \cite{ChenWentworth:21b}. Let $\Fcal$ be a holomorphic vector bundle over a smooth compact complex manifold $Y$ with $n-1$ K\"ahler forms $\omega_1, \cdots, \omega_{n-1}$. Then by \cite{Timorin:98}, we know $\omega_1, \cdots, \omega_{n-1}$ define a balanced metric $\omega$ through the following 
\begin{equation}
\omega^{n-1}=\omega_{1} \wedge \cdots
 \omega_{n-1}
\end{equation} 
i.e. $\omega$ is a Hermitian metric with $d\omega^{n-1}=0$. 
The Donaldson-Uhlenbeck-Yau theorem over complex manifolds with Gauduchon metrics gives (\cite{LiYau:87})
\begin{thm}\label{Smooth DUY}
Suppose $\Fcal$ is stable over $(X, \omega_1 \wedge \cdots \omega_{n-1}).$ There exists a Hermitian-Yang-Mills metric $H$ on $\Fcal$, i.e. 
$$\sqrt{-1} \Lambda_{\omega} F_H=\lambda \id$$ 
where $F_H$ is the Chern curvature of the metric $H$.
\end{thm}

The Hermitian-Yang-Mills equation implies
$$
(\frac{\sqrt{-1} }{2\pi}F_H- \frac{1}{r}\tr(\frac{\sqrt{-1} }{2\pi}F_H)\id) \wedge \omega_1 \wedge \cdots \omega_{n-1}=0
$$
where $r=\rank(\Fcal)$. By Timorin's results which generalize the classical Hodge-Riemann property for one K\"ahler form to multiple K\"ahler forms (see \cite{Timorin:98}), we know for any 
$1\leq i_1 < \cdots i_{n-2}\leq n-1$
the following holds pointwisely\begin{equation}\label{Hodge Riemann property}
-\tr((\frac{\sqrt{-1} }{2\pi}F_H- \frac{1}{r}\tr(\frac{\sqrt{-1} }{2\pi}F_H)\id)\wedge (\frac{\sqrt{-1} }{2\pi}F_H- \frac{1}{r}\tr(\frac{\sqrt{-1} }{2\pi}F_H)\id))\wedge \omega_{i_1} \wedge \cdots \omega_{i_{n-2}} \geq 0,
\end{equation}
where the equality holds if and only if 
$$\frac{\sqrt{-1} }{2\pi}F_H=\frac{1}{r}\tr(\frac{\sqrt{-1} }{2\pi}F_H)\id.$$
As a corollary, this generalizes the classical Bogomolov-Gieseker inequality to the multi-polarization setting
 \begin{cor}[Bogomolov-Gieseker inequality]
 Suppose $\Fcal$ is stable over $(X, \omega_1 \wedge \cdots \omega_{n-1})$. Then for any $1\leq i_1<\cdots i_{n-2}\leq n-1$
$$
(2rc_2(\Fcal)-(r-1)c_1(\Fcal))^2). [\omega_{i_1}]\cdots [\omega_{i_{n-2}}] \geq 0
$$
where the equality holds if and only if $\Fcal$ is projectively flat. 
 \end{cor}

\subsection{Chern classes and stability}\label{Chern classes}
In this section, we will give definitions of stable reflexive sheaves in the case of normal varieties. 

We fix $(X, \omega_1\wedge\cdots, \omega_{n-1})$ to be a normal K\"ahler variety with $n-1$ K\"ahler metrics.  Denote $\omega$ to be the balanced metric defined by
$$
\omega^{n-1}=\omega_1 \wedge \cdots \omega_{n-1}.
$$
Let $\Ecal$ be a reflexive sheaf over $X$. We will recall the notion of Chern classes we want to deal with.  We fix $p:\hat{X} \rightarrow X$ to be any  resolution. Recall in the introduction, we let $\Ebold_p$ be the space of reflexive sheaves on $\hat{X}$ which are isomorphic to $\Ecal$ away from the exceptional divisor. Pick $\hat{\Ecal}\in \Ebold_p$. Then the Chern numbers we need can be defined as  
\begin{equation}
\deg(\Ecal)= c_1(\hat \Ecal) p^*[\omega_1] \cdots p^*[\omega_{n-1}]
\end{equation}
and the slope of $\Ecal$ is defined as 
\begin{equation}\label{Slope definition}
\mu(\Ecal):=\frac{\deg\Ecal}{\rank\Ecal}.
\end{equation}
Given a connection $A$ on $\Ecal$ defined away from the singular set, if everything is smooth, then the Chern numbers above can be computed by using Chern-Weil theory. In the following, we will still denote $c_i(A)$ or $c_i(H)$ as the forms corresponding to $c_i(\Ecal)$,  if $A$ is the Chern connection given by $H$ on $\Ecal$.

As a direct corollary of Equation (\ref{n-1 forms}), we have 

\begin{prop}\label{Well-definedness of slope}
$\deg(\Ecal)$ is independent of the choice of the resolutions $p$ and $\Ebold_p$.
\end{prop}

\begin{proof}
 Suppose $\hat{\Ecal}\in \Ebold_{p_1}$ and $\hat{\Ecal}'\in \Ebold_{p_2}$ are two such  extensions for two different resolutions $p_1: \hat{X}^1 \rightarrow X$ and $p_2: \hat{X}^2 \rightarrow X$. Let $p: \hat{X} \rightarrow X$ be a resolution so that the following diagram commutes 
\[\begin{tikzcd}
	& {\hat{X}} \\
	{\hat{X}_1} && {\hat{X}_2} \\
	& X
	\arrow["{q_1}"', from=1-2, to=2-1]
	\arrow["{q_2}", from=1-2, to=2-3]
	\arrow["{p_1}"', from=2-1, to=3-2]
	\arrow["{p_2}", from=2-3, to=3-2]
	\arrow["p", from=1-2, to=3-2]
\end{tikzcd}\]
for some map $q_1, q_2$. Here $\hat{X}$ for example can be taken to be the resolution of the fiber product of $\hat{X}_1$ and $\hat{X}_2$ over $X$, which again is a resolution of $X$. Take $\hat{\Ecal}\in \Ebold_p$. By equation \ref{n-1 forms} applied to $p_1$, we can write 
$$
p_1^*\omega_1  \wedge \cdots p_1^*\omega_{n-1}=\Omega^0_{n-1}+d\Phi_{n-1}
$$
over $\hat{X}^1$ where $\Omega_{n-1}^0$ is a compact supported form over $X\setminus Z$ and $Z=\text{Sing}(\hat{\Ecal})\cup X^s$. Then we know 
$$
\deg(\hat{\Ecal})=\int_{\hat{X}_1\setminus p_1^{-1}(Z)} c_1(\hat\Ecal_1) \wedge \Omega^0_{n-1}.
$$
On the other hand, by definition, we know $c_1(\hat{\Ecal}_1)|_{X\setminus Z}=c_1(\hat{\Ecal})|_{X\setminus Z}$, thus 
$$
\int_{\hat{X}_1\setminus p_1^{-1}(Z)} c_1(\hat\Ecal_1) \wedge \Omega^0_{n-1}=\int_{\hat{X}_1\setminus p_1^{-1}(Z)} c_1(\hat\Ecal) \wedge \Omega^0_{n-1}.
$$
Since $[\Omega_{n-1}^0]$ can be naturally viewed as the same de Rham class over $\hat{X}$ as $[(p)^*\omega_{i_1}]\wedge \cdots [(p)^*\omega_{i_{n-2}}]$, we have 
$$
\int_{\hat{X}_1\setminus p^{-1}_1(Z)} c_1(\hat\Ecal') \wedge \Omega^0_{n-1}=c_1(\hat{\Ecal}) .[p^*\omega_{1}].\cdots [p^*\omega_{n-1}]=\deg(\hat{\Ecal})
$$
which implies $\deg(\hat{\Ecal}_1)=\deg(\hat{\Ecal})$. Similarly $\deg(\hat{\Ecal}_2)=\deg(\hat{\Ecal})$. The conclusion follows.
\end{proof}

In particular, the following slope stability is well-defined
\begin{defi}
$\Ecal$ is called slope stable (resp. semistable) if for any $\Fcal\subset \Ecal$, $\mu(\Fcal)<\mu(\Ecal)$ (resp. $\mu(\Fcal) \leq \mu(\Ecal)$). $\Ecal$ is called polystable if it is a direct sum of stable ones.
\end{defi}

The following fact will be used in later sections which we include a proof for completeness
\begin{lem}\label{Nontrivial map must be isomorphism}
Suppose $\Ecal_1$ and $\Ecal_2$ are two stable reflexive sheaves with the same slope. Then any nontrivial map between $\Ecal_1$ and $\Ecal_2$ is an isomorphism. In particular, a stable reflexive sheaf is simple.
\end{lem}
\begin{proof}
Suppose $\phi$ is such a map. By stability, we must have $\Ker\phi=0$, thus $\phi$ is injective. Then we have 
$$
0\rightarrow \det(\Ecal_1) \xrightarrow{\det \phi} \det(\Ecal_2)\rightarrow \tau_{D} \rightarrow 0
$$
for some $\tau_{D}$ which is a torsion sheaf supported on a subvariety $D\subset X$. By definition, we have
$$
\deg(\Ecal_2)-\deg(\Ecal_1)=c_1(\tau_D).[\omega_1]\cdots[\omega_{n-1}].
$$
Now any pure codimension one component of $D$  will contribute strictly positively to the equality above and give a contradiction, thus $\text{codim}_{\CBbb} D \geq 2$. The conclusion follows. 
\end{proof}
\subsection{Discriminant}
We need another important concept in our setting, i.e. the so-called \emph{discriminant}. As we mentioned in the introduction, due to the fact that $X$ is a normal variety where in general the singularities could contribute in an essential way to higher Chern classes, unlike the slope, it could not be defined by using a single resolution. Instead, motivated by the gauge theoretical picture in the proof of the Donaldson-Uhlenbeck-Yau theorem in our setting, we include all the resolutions. For this, as in the introduction, given a sheaf $\Hcal$, we denote 
$$
\Delta(\Hcal)=2rc_2(\Hcal)-(r-1)c_1(\Hcal)^2
$$
when the chern classes are well-defined in the ordinary sense. Now for any $1\leq i_1<\cdots i_{n-2}\leq n-1$,  the discriminant is defined as 
\begin{equation}\label{BG chern numbers}
\begin{aligned}
&\Delta(\Ecal)
 [\omega_{i_1}]\cdots [\omega_{i_{n-2}}]\\
 =&\inf_{\text{ K\"ahler } p} \inf_{\hat{\Ecal}\in \Ebold_p} \Delta(\hat{\Ecal})
 p^*[\omega_{i_1}] \cdots p^*[\omega_{i_{n-2}}].
 \end{aligned}
\end{equation}
Now we discuss typical cases where this could be simplified and computed. 

\subsubsection{Locally free sheaves over projective normal surface}
When $\Ecal$ is locally free,  $\Ecal$ defines a topological bundle over $X$ which is a CW complex. Thus, the Chern classes of $\Ecal$ are naturally defined in the singular cohomology as well as the corresponding C\v ech cohomology. We will abuse notation for not making a difference between the two cohomologies. In particular, the term we defined above can be computed using the cup products in singular cohomology.
\begin{lem}
Assume $\Ecal$ is locally free over a normal surface. Then $\Delta(\Ecal)$
coincides with the topological definition.
\end{lem}

The proof of this follows directly from the following inequality that has independent interest 

\begin{prop}\label{Pull-back minimizes}
For any $\hat{\Ecal}\in \Ebold_p$ and any resolution $p: \hat{X} \rightarrow X$, the following holds
$$
\Delta(\Ecal)\cap X \leq \Delta(\hat{\Ecal}) \cap \hat{X}
$$
\end{prop}

\begin{proof}
Passing to a further resolution, we can assume that for any $x\in \hat{X}$, the exceptional divisors are simple normal crossings. We first observe that  
$$
c(\hat{\Ecal})=(1+p^*x_1+e_1)\cdots (1+p^*x_m+e_m)
$$
where 
$$
c(\Ecal)=(1+x_1)\cdots (1+x_m)
$$
and $e_i$ can be represented by the exceptional divisors. In particular, by computation, we have 
$$
(\Delta(p^*\Ecal)-\Delta(\hat\Ecal))\cap \hat{X}=\sum_{i<j}(e_i-e_j)^2\leq 0
$$
where 
\begin{itemize}
\item for the first equality, we used the observation 
$$
e_t. p^*x_s=0
$$
since $p^*x_s|_{X^s}=0$ due to dimensional reasons.
\item for the second inequality, this follows from a lemma of Du Val (\cite[Proposition 2.1.12]{Nemethi:22}) which states the natural intersection matrix given by intersecting irreducible components of the exceptional divisor is negative definite. 
\end{itemize}
The conclusion follows.
\end{proof}

\subsubsection{On the minimal resolution of a normal surface}\label{On the minimal resolution of a normal surface} Now let $p^{min}: \hat{X} \rightarrow X$ be the minimal resolution, i.e. for any other resolution $p':\hat{X}' \rightarrow X$, it factors through a unique map $q: \hat{X}' \rightarrow \hat{X}$ so that the following diagram commutes
\[\begin{tikzcd}
	{\hat{X}'} && {\hat{X}} \\
	& X
	\arrow["p'"', from=1-1, to=2-2]
	\arrow["{p^{min}}", from=1-3, to=2-2]
	\arrow["q", from=1-1, to=1-3]
\end{tikzcd}\]
i.e. $p'=p^{min}\circ q$.

\begin{prop}\label{Prop2.12}
The discriminant can be computed on the minimal resolution, i.e. 
$$
\Delta(\Ecal)=\inf_{\hat{\Ecal}\in \Ebold_{p^{min}}} \Delta(\hat{\Ecal}).
$$
\end{prop}

\begin{proof}
As above, fix any resolution $p':\hat{X}'\rightarrow X$, it factors as $p'=p^{min}\circ q$. Given any $\hat{\Ecal}'\in \Ebold_{p'}$, then $(q_*\hat{\Ecal}')^{**}\in \Ebold_{p^{min}}$. It suffices to prove that 
$$
\Delta((q_*\hat{\Ecal}')^{**})\leq \Delta(\hat{\Ecal}').
$$
Indeed, we know 
$$
\Delta((q_*\hat{\Ecal}')^{**})=\Delta(q^*((q_*\hat{\Ecal}')^{**})) \leq \Delta(\hat{\Ecal}')
$$
where the first equality follows from definition and the second equality follows from Proposition \ref{Pull-back minimizes}.
\end{proof}

\subsubsection{When $X$ is smooth in codimension two}\label{Codimension three case}
\begin{prop}\label{Prop2.13}
Assume $X$ is smooth in codimension two, the discriminant $\Delta(\Ecal)
 [\omega_{i_1}]\cdots [\omega_{i_{n-2}}]$ can be computed by using any resolutions. Furthermore, if $[\omega_{1}], \cdots [\omega_{n-1}]$ can be represented by very ample Cartier divisors, then 
$$
\begin{aligned}
\Delta(\Ecal).[\omega_{i_1}] \cdots  [\omega_{i_{n-2}}]=
(2rc_2(\Ecal|_{D_{i_1}\cap \cdots D_{i_{n-2}}})-(r-1)c_1(\Ecal|_{D_{i_1}\cap \cdots D_{i_{n-2}}})^2) 
\end{aligned}
$$
where for each $i$, $[D_{i}]=[\omega_i]$ and $D_{i_1}\cap \cdots D_{i_{n-2}}$ is a smooth surface in $X$.
\end{prop}

\begin{proof}
Suppose $\hat{\Ecal}\in \Ebold_p$ and $\hat{\Ecal}'\in \Ebold_{p'}$ for two different resolutions $p: \hat{X}\rightarrow X$ and $p': \hat{X}'\rightarrow X$. Similar to Proposition \ref{Well-definedness of slope}, it suffices to prove the case when we have the following commutative diagram 
\[\begin{tikzcd}
	{\hat{X}} && {\hat{X}'} \\
	& X
	\arrow["q", from=1-1, to=1-3]
	\arrow["p'"', from=1-1, to=2-2]
	\arrow["{p}", from=1-3, to=2-2]
\end{tikzcd}\]
By Equation (\ref{n-2 forms}) applied to $p$, we can write 
$$
p^*\omega_{i_1} \wedge \cdots p^*\omega_{i_{n-2}}=\Omega^0_{n-2}+d\Phi_{n-2}
$$
over $\hat{X}'$ where $\Omega^0_{n-2}$ is a smooth $2n-4$ forms with compact support in $X^{reg}$, and $\Phi_{n-2}$ is a smooth $2n-4$ form over $\hat{X}'$. Then 
$$
\begin{aligned}
\Delta(\hat{\Ecal}).[\omega_{i_1}] \cdots  [\omega_{i_{n-2}}]=&\int_{X\setminus Z}\Delta(\hat{\Ecal})\wedge \Omega^0_{n-2}\\
=&\int_{X\setminus Z} \Delta(\hat{\Ecal}')\wedge \Omega^0_{n-2}\\
=&\Delta(\hat{\Ecal}').[(p')^*\omega_{i_1}] \cdots  [(p')^*\omega_{i_{n-2}}]
\end{aligned}
$$
where the first equality follows from definition, the second inequality follows from $\Delta(\hat{\Ecal})=\Delta(\hat{\Ecal}')$ over $X\setminus Z$, and the last equality follows from  $[(p')^*\omega_{i_1}] \cdots  [(p')^*\omega_{i_{n-2}}]$ and $[\Omega^0_{n-2}]$ can be naturally viewed as the same de Rham class over $\hat X$. For the second part, it suffices to notice that viewed as linear functionals on $H^4_{dR}(\hat{X}, \RBbb)$, $D_{i_1}\cap \cdots D_{i_{n-2}}$ and $\Omega_{n-2}^0$ define the same element.
\end{proof}

\subsection{Admissible Hermitian-Yang-Mills metrics and regularity results}\label{Regularity results}
We use the following notion of admissible Hermitian-Yang-Mills metric (\cite{BandoSiu:94})
\begin{defi}
An admissible Hermitian-Yang-Mills metric is defined as a smooth Hermitian metric $H$ on a holomorphic bundle $F$ over $X\setminus Z$ where $X_s\subset Z$ and $Z\setminus X_s$ is a subvariety of $X^{reg}$ of codimension at least two, and the metric $H$ satisfies the following 
\begin{itemize}
\item $\int_{X\setminus Z} |F_H|^2 < \infty$;
\item $\sqrt{-1}\Lambda_{\omega} F_{H}=\lambda \id$
where $\lambda$ is usually referred as the Einstein constant
\end{itemize}
where $F_H$ denotes the curvature. The associated Chern connection $A$ is usually referred as an (admissible) Hermitian-Yang-Mills connection. 
\end{defi}

We have the following essentially due to Bando and Siu (see \cite[Theorem $2$]{BandoSiu:94} and \cite[Proposition $46$]{ChenWentworth:21a})
\begin{thm}\label{Extension}
An admissible Hermitian-Yang-Mills connection $A$ defines a reflexive sheaf $\Fcal$  over $X^{reg}$. Furthermore, the admissible Hermitian-Yang-Mills metric could be extended and defined wherever $\Fcal$ is locally free.
\end{thm}

Now an adaption of the slicing argument in \cite{BandoSiu:94} using one dimensional slices instead of two dimensional slices gives the following regularity result. We include the proof for completeness here.

\begin{prop}
For any local section $s$ of $F$, $\log^{+}|s|^2\in W^{1,2}_{loc}$.
\end{prop}

\begin{proof}
We will use $Q_1 \lesssim (\gtrsim) Q_2$ to denote $Q_1\leq (\geq )C Q_2$ for some constant $C$. By Corollary \ref{Metric Equivalence}, we can assume the metric is induced from the flat metric on $\CBbb^N$. The statement is local. Fix $z\in Z$. We can assume $z\in X\subset U\subset \CBbb^N$ with $\omega$ being the restriction of the standard flat metric. By shrinking $U$, we can assume $U=B^l \times B^{l'}$ and $z=0$
$$
p: X\cap U \rightarrow B^{l'}
$$
and $p^{-1}(t)$ is a smooth curve while $p^{-1}(0)$ is a smooth curve with a point singularity at the origin. By the assumption on $Z$, we know $p^{-1}(t)\cap Z=\emptyset$ for generic $t$.  Let $u_t=\log^{+}|s|^2|_{p^{-1}(t)}$ which is smooth for generic $t$. We have 
$$
\Delta_t u_t \gtrsim -|F_{H}|.
$$
Let $\chi$ be a cut-off function on $B^{l'}$ supported near $0$. Then by doing integration by parts and applying Cauchy-Schwartz inequality, we have 
$$
\int_{p^{-1}(t)} |\nabla^t (\chi u_t)|^2 \lesssim  \delta \int_{p^{-1}(t)} (\chi u_t)^2+\delta^{-1}\int_{p^{-1}(t)} \chi^2 |F_H|^2+\int_{p^{-1}(t)} |\nabla^t \chi|^2 u_t^2. 
$$
for any fixed $0<\delta<<1$. With the induced metric on $p^{-1}(t)$, by \cite[Theorem $18.6$]{Simon:83}, there exists a Poincar\'e inequality for $p^{-1}(t)$ and the Sobolev constant is independent of $t$ since $p^{-1}(t)$ is a minimal submanifold of $U$. So we have 
$$
\int_{p^{-1}(t)} |\nabla^t (\chi u_t)|^2 \lesssim  \delta^{-1}\int_{p^{-1}(t)} \chi^2 |F_H|^2+\int_{p^{-1}(t)} |\nabla^t \chi|^2 u_t^2
$$
which further implies 
$$
\int_{p^{-1}(K)} |\nabla' (\chi u_t)|^2 \lesssim  \delta^{-1}\int_{p^{-1}(K)} \chi^2 |F_H|^2+\int_{p^{-1}(K)}  |\nabla' \chi|^2 u_t^2
$$
for some compact set $0\in K \subset p(X\cap U)$. Here $\nabla'$ denotes the derivative in the fiber direction. Now choose $(n-1)$ more such projections and add them together. This gives the bound we need.
\end{proof}

\begin{cor}\label{L infinity bound}
For any global section $s$ of $F$, $\log^+|s|^2\in L^\infty$.
\end{cor}

\begin{proof}
 By \cite{Simon:83}, we know $X$ admits a Sobolev inequality for compact supported functions over $X^{reg}$ if $X$ is endowed with a fixed K\"ahler metric  i.e. 
$$
\|f\|_{L^{\frac{2n}{n-1}}} \leq C (\|f\|_{L^2(X)}+\|\nabla f\|_{L^2(X)})
$$
for any $f\in C^1_{c}(X^{reg})$. More precisely, locally we can assume the K\"ahler metric is induced from the flat metric on $\CBbb^N$. In particular, $X$ is a minimal submanifold of $\CBbb^N$ away from the singular set, thus by \cite[Theorem $18.6$]{Simon:83}, there exists a Sobolev inequality for $X$. Now the approximation result in  \cite[Section $4$]{LiTian1995} implies a Sobolev inequality for functions in $W^{1,2}(X^{reg})$ where as pointed out in \cite[Section $0$]{LiTian1995}, the essential argument does not require the the variety to be projective. By Corollary \ref{Metric Equivalence}, we also have a Sobolev inequality for $(X,\omega_1\wedge \cdots \omega_{n-1})$.  Given this, we can apply the Moser iteration to $\log^+|s|^2$ which satisfies 
$$
\Delta \log^+|s|^2 \geq \lambda_\infty
$$
where $\lambda_\infty$ is the Einstein constant of $A_\infty$, thus $\log^{+}|s|^2\in L^\infty$. This concludes the proof. 
\end{proof}

\begin{rmk}
We also refer the readers to \cite{GPST:23} for the Sobolev inequalities built for a more general class of K\"ahler spaces. 
\end{rmk}

\begin{cor}\label{HE uniqueness}
The admissible Hermitian-Yang-Mills connection on $F$ is unique if it exists. Furthermore, if $\Fcal$ can be extended to be a stable reflexive sheaf over $X$, the Hermitian-Yang-Mills metric is unique up to scaling.
\end{cor}
\begin{proof}
Otherwise, suppose we have two such metrics $H_1$ and $H_2$  with Einstein constants $\lambda_1\geq \lambda_2$. Consider the identity map $\id\in \Hom(F, F)$ with the endowed metric $H_1^*\otimes H_2$. Then straightforward computation shows
$$
\Delta |\id|^2 =2 |\nabla \id|^2-(\lambda_2-\lambda_1)|\id|^2. 
$$
Since $\id \in L^\infty$ by Corollary \ref{L infinity bound}, we can do integration by parts to conclude $\nabla \id=0$, i.e. $\id$ is parallel and $\lambda_2=\lambda_1$. The uniqueness follows. If $F$ comes from a stable reflexive sheaf over $X$, write 
$$
H_1(.,.)=H_2(g., .)
$$ where $g$ is Hermitian with respect to $H_2$. The fact that $\id$ is parallel implies $g$ is holomorphic, thus a multiple of the identity map by Lemma \ref{Nontrivial map must be isomorphism}.  
\end{proof}

\subsection{Hermitian-Yang-Mills metrics using perturbations}
The following has been observed in the K\"ahler case (see \cite{CGG:21} \cite{Wu:21}) by using the boundedness result in \cite{Toma:19}. Similar argument also works in our setting. Take a resolution $p:\hat X \rightarrow X$ and let $\hat{\Ecal}:=(p^*\Ecal)^{**}$.
\begin{prop}\label{Openess of stability}
Suppose $\Ecal$ is a stable reflexive sheaf over $(X, \omega_1 \wedge \cdots \omega_{n-1})$, then $\hat \Ecal$ is stable over $(\hat X, (p^*\omega_1+i^{-1} \theta) \wedge \cdots (p^*\omega_{n-1}+i^{-1} \theta) )$ for $i>>1$.
\end{prop}

\begin{proof}
By passing to a further resolution (\cite[Theorem $3.5$]{Ross:68}), it suffices to prove it when $\hat{\Ecal}$ is locally free. By definition, $\hat{\Ecal}$ is stable with respect to $p^*\omega_1 \wedge \cdots p^*\omega_{n-1}$. Otherwise,  the push-forward of the destabilizing subsheaf will destabilize $\Ecal$ as well. Now we argue by contradiction for the main statement. By passing to a subsequence, we have a sequence of quotient maps $q_i: \hat{\Ecal}\rightarrow  \Fcal_i$ and $\mu_i(\Fcal_i)\leq \mu_i(\hat\Ecal)$ where $\mu_i$ denotes the slope of a sheaf with respect to the metric $(p^*\omega_1+i^{-1} \theta) \wedge \cdots (p^*\omega_{n-1}+i^{-1} \theta)$. By choosing a metric on $\hat{\Ecal}$, an easy computation shows 
$$c_1(\Fcal_i)[\theta]^{l} [p^*\omega_{k_1}] \cdots [p^*\omega_{k_{n-l-1}}]\geq C$$
for some $C$ independent of $l,k_1,\cdots k_{l-1}$. In particular, this implies 
$$
\mu(\Fcal_i)+O(\frac{1}{i}) \leq \mu_i(\hat\Ecal).
$$
By the boundedness result in \cite[Corollary $6.3$]{Toma:19} applied to the degree defined by $\omega_1 \wedge \cdots \omega_{n-1}$, we can assume $$
p_*q_i: \Ecal \rightarrow \Img(p_*q_i)
$$ 
lies in the same component of the corresponding Douady space and has a limit over $X$. Thus we get a quotient map $q_\infty: \Ecal \rightarrow \Fcal_\infty$ where $\mu(\Fcal_\infty)\leq \mu(\Ecal)$. This contradicts the stability of $\Ecal$.
\end{proof}

Now we will fix a resolution $p:\hat{X} \rightarrow X$ so that $\hat{\Ecal}:=(p^*\Ecal)^{**}$ is locally free. From the above, for any $i>>1$, by the Donaldson-Uhlenbeck-Yau theorem for multipolarizations (see Theorem \ref{Smooth DUY}), there exists a Hermitian-Yang-Mills metric $H_i$ on $\hat \Ecal$ over $(\hat X, (p^*\omega_1+i^{-1} \theta) \wedge \cdots (p^*\omega_{n-1}+i^{-1} \theta))$. In particular, 
\begin{equation}\label{HE equation for perturbation}
\begin{aligned}
&(\sqrt{-1} F_{H_i}-\frac{\tr(\sqrt{-1} F_{H_i})}{\rank \Ecal}\id) \wedge  (p^*\omega_1+i^{-1} \theta) \wedge \cdots (p^*\omega_{n-1}+i^{-1} \theta)=0
\end{aligned}
\end{equation}
i.e. 
$$
\sqrt{-1} F_{H_i}-\frac{\tr(\sqrt{-1} F_{H_i})}{\rank \Ecal}\id
$$
is primitive with respect to $ (p^*\omega_{k_1}+i^{-1} \theta) \wedge \cdots (p^*\omega_{k_{n-2}}+i^{-1} \theta)$ for any $1\leq k_1<\cdots k_{n-2}\leq n-1$. Now the Hodge-Riemann property applied to $(p^*\omega_{i_1}+i^{-1} \theta) \wedge \cdots (p^*\omega_{i_{n-2}}+i^{-1} \theta)$  (see Equation (\ref{Hodge Riemann property})) implies the following Bogomolov-Gieseker inequality 
\begin{equation}\label{GB inequality for perturbation}
\int_{\hat X} (2rc_2(A_i)-(r-1)c_1^2(A_i)) \wedge  (p^*\omega_{i_1}+i^{-1} \theta) \wedge \cdots (p^*\omega_{i_{n-2}}+i^{-1} \theta) \geq 0
\end{equation}
where the inequality holds if and only if $A_i$ is projectively flat.

Let $\hat E$ be the underlying smooth bundle of $\hat{\Ecal}$. By doing complex gauge transforms (\cite[Section 5]{UhlenbeckYau:86}), we have a sequence of Hermitian-Yang-Mills connections $A_i$ on a fixed unitary bundle $(\hat E, H)$. Then we can take a Uhlenbeck limit using gauge theory (see \cite{ChenWentworth:21a}). Namely, by passing to a subsequence, up to gauge transforms, we can assume $A_i$ converges locally smoothly to a limiting connection $A_\infty$ defined on $E|_{X\setminus \Sigma}$ where $\Sigma$ is the so-called bubbling set defined as 
\begin{equation}
\Sigma=\{z\in X^{reg}: \lim_{r\rightarrow 0^+} \liminf_{i\rightarrow \infty} r^{4-2n}\int_{B_r(z)}|F_{A_i}|^2\dvol_i>0\}.
\end{equation}
Here the base is given by $X\setminus \Sigma$ together with a sequence of metircs that converge to the metric given by the multiple K\"ahler forms, thus Uhlenbeck compactness applies.

By \cite[Proposition $46$]{ChenWentworth:21c} (see also \cite[Theorem $4.3.3$]{Tian:00} for the K\"ahler case), we know
\begin{lem}
$\Sigma \subset X^{reg}$ is a subvariety of codimension at least two.
\end{lem}

\begin{rmk}
As we will see later, $\Sigma$ is actually empty. 
\end{rmk}

We need the following key observation regarding the holomorphic vector bundle $\Ecal_\infty$ over $X^{reg}$ defined by $A_\infty$.
\begin{prop}\label{Slope equality}
Assume $\Ecal_\infty$ can be extended to be a reflexive sheaf over $X$. Then $\mu(\Ecal_\infty)=\mu(\Ecal)$. In particular, the admissible Hermitian-Yang-Mills metric computes the slope of $\Ecal_\infty$.
\end{prop}

\begin{proof}
Denote $\hat \Ecal_\infty =(p^*\Ecal_\infty)^{**}$. By Equation \ref{n-1 forms}, we can write 
$$
p^*(\omega_1  \wedge \cdots\omega_{n-1})=\Omega^0_{n-1}+d\Phi_{n-1}
$$
where $\Omega^0_{n-1}$ is compact supported in $\hat{X} \setminus p^{-1}(X^s)$, which is possible because $X^s\subset X$ is a subvariety of codimension at least two. Given this, we have the following
$$
\begin{aligned}
\mu(\Ecal_\infty)
&= \frac{\int_{\hat X} c_1(\hat{\Ecal}_\infty) \wedge p^*(\omega_1  \wedge \cdots\omega_{n-1})}{\rank \hat{\Ecal}_\infty}\\
&=\frac{\int_{\hat X} c_1(\hat{\Ecal}_\infty) \wedge (\Omega_{n-1}^{0}+d\Phi_{n-1})}{\rank \hat{\Ecal}_\infty}\\
&=\frac{\int_{\hat X} c_1(\hat{\Ecal}_\infty) \wedge \Omega_{n-1}^{0}}{\rank \hat{\Ecal}_\infty}\\
&=\frac{\int_{\hat X} c_1(A_\infty) \wedge \Omega_{n-1}^{0}}{\rank \hat{\Ecal}_\infty}\\
&=\lim_i \frac{\int_{\hat X} c_1(A_i) \wedge \Omega_{n-1}^{0}}{\rank \hat{\Ecal}_\infty}\\
&=\lim_i \frac{\int_{\hat X} c_1(A_i) \wedge p^*(\omega_1  \wedge \cdots\omega_{n-1})}{\rank \hat{\Ecal}_\infty}\\
&=\lim_i \frac{\int_{\hat X} c_1(A_i) \wedge (p^*(\omega_1)+i^{-1}\theta)  \wedge \cdots (p^*(\omega_{n-1})+i^{-1}\theta)}{\rank \hat{\Ecal}_\infty}\\
&=\lim_i \mu_i(\hat \Ecal)
\end{aligned}
$$
All the equalities are straightforward except the fourth and fifth one: the fourth one follows from the fact that $c_1(A_\infty)$ is smooth over $X^{reg}$ by Theorem \ref{Extension} applied to $\det(\Ecal_\infty)$ and $\Omega_{n-1}^0$ is compact supported away from the exceptional divisors; the fifth one follows from the  strong $L^1$ convergence of  $F_{A_i}$ over the support of $\Omega_{n-1}^0$.  Combined with that the Einstein constant of $A_i$ converges to the Einstein constant of $A_\infty$, we know the admissible Hermitian-Yang-Mills metric computes the slope of $\Ecal_\infty$.
\end{proof}

\subsection{Nontrivial limits of holomorphic sections away from the singular set} In this section, we will prove a technical result needed later. By Corollary \ref{Metric Equivalence}, we can assume $X$ is endowed with a K\"ahler metric. 

The convergence of the Hermitian-Yang-Mills connections in the previous sections can be now described by saying we have a sequence of connections, up to gauge transforms, converging locally smoothly over $X\setminus Z$ where $Z=X^s \cup \Sigma$ as defined before. Here we fix the original metric on $X$ given by $\omega_1 \wedge \cdots \omega_{n-1}$ to look at the convergence. Namely, $A_i$ is a sequence of unitary connections on a unitary bundle $(\hat{\Ecal}, H)$ satisfying
\begin{itemize}
\item $F_{A_i}^{0,2}=0$;
\item up to gauge transforms, $A_i$ converges to $A_\infty$ locally smoothly over $X\setminus Z$.
\end{itemize}
For any $\epsilon>0$, define $Z^\epsilon$ to be a closed $\epsilon$-neighborhood of $Z$ in $X$. Furthermore, $Z^\epsilon \subset \text{Interior}(Z^{\epsilon'})$ if $\epsilon< \epsilon'$ and $Z^\epsilon$ converges to $Z$ as $\epsilon \rightarrow 0$. Suppose we have a sequence of holomorphic sections $\{s_i\in H^0(X\setminus Z, \Ecal_i)\}$ with  $\|s_i\|_{L^2(X_\epsilon)}=1$ where $X_\epsilon=X\setminus Z^\epsilon$. Standard elliptic theory for holomorphic sections guarantees the existence of a limit over $X^\epsilon$.  The goal is to show the limit is actually nontrivial using similar argument in \cite{ChenSun:18}. 
\subsubsection{Curve property} 
\begin{defi}
A closed subset $S \subset X$ admits a \emph{good cover} if $X$ can be covered by finitely many open sets $U_{k}\subset \CBbb^N$ where $\omega|_{U_k\cap X}=\sqrt{-1}\partial \bar{\partial} \rho|_{X\cap U_k}$ and $\rho$ is a strongly smooth pluri-subhamronic function on $\overline{U_k}$ such that
\begin{itemize}
\item $ U_k =B^{l}_{\delta^k_2}\times B_{\delta^k_3}^{l'}$ for some $\delta_2^k,\delta_3^k>0$, where $B^l_{\delta^k_2}$ denotes the ball $\{|z|<\delta^k_{2}\}$ in $\CBbb^l$ and $B^{l'}_{\delta^k_3}$ denote the ball $\{|z|<\delta_3^k\}$ in $\CBbb^{l'}$. Here $l+l'=N$.
\item $ \overline{U_k}\cap S \subset V_k =B^l_{\delta^k_1}\times \overline{B^{l'}_{\delta^k_3}}$ for some $0<\delta^k_1< \delta^k_2$ where $\overline{U_k}=\overline{B^{l}_{\delta^k_2}}\times \overline{B_{\delta^k_3}^{l'}}$.
\item for any $z\in \overline{B^{l'}_{\delta_3^k}}\cap \rho_k (X\cap \overline{U_k})$, $\rho_k^{-1}(z)\cap X$ is a smooth surface away from finitely many point. Here $\rho_k: B^l_{\delta_1^k} \times \overline{B^{l'}_{\delta_3^k}} \rightarrow \overline{B^{l'}_{\delta_3^k}}$ denotes the natural projection.
\end{itemize}
\end{defi}

\begin{lem}\label{lem2.10}
$Z\subset X$ admits a good cover. 
\end{lem}

\begin{proof}
For any $z\in X$, since $X_s$ is a codimension $2$ complex subvariety of $X$ and $\Sigma$ is a subvariety of $X^{reg}$ of codimension at least two, locally near $z$, we can cover $X$ with an open set $U$ in $\CBbb^N$ by assuming $z=0$ so that for some orthogonal projection $\rho: U\subset \CBbb^N=\CBbb^l\times \CBbb^{l'} \rightarrow \CBbb^{l'}$, $\rho^{-1}(y)\cap X\cap B^{l}_{\delta_2}$ is a smooth surface away from $\rho^{-1}(y)\cap X_s\cap B^l_{\delta_2}$ which consists of finitely many points for any $y\in \rho(X\cap U)$ for some $\delta_3>0$; furthermore, $\rho^{-1}(y) \cap \Sigma\cap B^{l}_{\delta_2}$ consists of points which accumulate at most near $X_s\cap B^l_{\delta_2}$. Then near $z$, we can easily construct a neighborhood  $U_p=B^{l}_{\delta_2}\times B^{l'}_{\delta_3}$ for some $\delta_2, \delta_3>0$ and $\overline{U_p}\cap \Sigma \subset V_p$ where $V_p=B^l_{\delta_1}\times \overline{B^{l'}_{\delta_3}}$ for some $0<\delta_1<\delta_2$. Now we get such an open cover of $X$ given by $\cup_p U_p$. Since $X$ is compact, we can take a finite subcover $\cup_k U_{p_k}$ which gives a good cover for $Z\subset X$.
\end{proof}

\begin{cor}\label{good cover}
For $0<\epsilon<<1$, $Z^{\epsilon}\subset X$ admits a good cover.
\end{cor}
\begin{proof}
Otherwise, suppose the statement is false for a sequence $\epsilon_i\rightarrow 0^+$. By definition, $Z^{\epsilon_i}$ converges to $Z$. Let $\cup_k U_k$ be a good cover for $Z$. We want to show it is a good cover for $Z^{\epsilon_i}$ for $i$ large, thus get a contradiction. We only need to verify that $\overline{U_k}\cap Z^{\epsilon_i}\subset V_k$ for $i$ large. Otherwise, by passing to a subsequence and using the finiteness of $\{U_k\}_k$, we can assume for some fixed $k$, there always exists $z_{i} \in (\overline{U_k}\cap Z^{\epsilon_i})\setminus V_k$ for each $i$ and $z_i$ converges to $z\in \overline{U_k}\cap Z$. Then $z\in V_k$ and thus $z_i\in V_k$ for $i$ large. Contradiction.
\end{proof}

\begin{lem}\label{Curve property}
For any $\epsilon>0,$ let $\cup_kU_k$ be a good cover of $Z^{\epsilon}\subset X$. Then there exists a constant $C=C(\epsilon)>0$ so that for any $z\in X\setminus Z^{\frac{\epsilon}{2}}  $, there exists a curve  $D_z \subset U_k\cap X$ for some $k$ such that $\partial D_z\subset X\setminus Z^{\epsilon}$, $d(\partial D_z, \partial (X\setminus Z^\epsilon))\geq C$ and 
$d(D_z, Z)\geq C.$ Here $d$ denotes the distance between two sets. 
\end{lem}

\begin{proof}
For each $k$, $\overline{U_k} \cap Z \subset \overline{U_k}\cap Z^{\frac{\epsilon}{2}}\subset \overline{U_k} \cap Z^{\epsilon} \subset V_k$  where $\overline{U_k} =\overline{B^{l}_{\delta^k_2}}\times \overline{B_{\delta^k_3}^{l'}}$ and $V_k=B^l_{\delta^k_1}\times \overline{B^{l'}_{\delta^k_3}}$. Consider the projection $\rho_k: \overline{U_k} \rightarrow \overline{B^{l'}_{\delta^k_3}}$.  By assumption, we have
$$(\overline{B^l_{\delta^k_2}}\times \overline{B^{l'}_{\delta^k_3}}) \cap Z = \overline{U_k} \cap Z \subset B^l_{\delta^k_1}\times \overline{B^{l'}_{\delta^k_3}}$$
which implies $\rho_k^{-1}(y)\cap Z \cap B^{l}_{\delta_2^k}$ is a compact subset of $B^l_{\delta_1^k}$ for any $y\in \rho_k(X\cap \overline{U_k})$. 
Also we know $\rho_k^{-1}(y)\cap X_s \cap B^{l}_{\delta_2^k}$ is a complex subvariety of $B^{l}_{\delta_2^k}$, thus consists of finitely many points. Since  $((\overline{B^l_{\delta^k_2}}\times \overline{B^{l'}_{\delta^k_3}}) \cap Z \setminus X_s) \cap \rho^{-1}(y)$ is a subvariety of $B^l_{\delta_1^k}\setminus X_s$, we know it consists of finitely many points which accumulate at most near $X_s\cap B^l_{\delta_1^k}$.
Now for any $z \in \overline{X\setminus Z^{\frac{\epsilon}{2}}}$, suppose $z\in U_k$, then $\rho^{-1}_k(\rho_k(z))\cap Z\cap B^l_{\delta_2^k}$ consists of points which accumulate mostly  near $Z_s\cap B^{l}_{\delta_1}$. As a result, one can easily find a curve $D_z \subset U_k$ containing $z$ such that $D_z \cap Z=\emptyset$ and $\partial D_z \subset (U_k\setminus V_k) \cap X \subset U_k \cap X \setminus Z^{\epsilon}$.  By perturbing the curve $D_z$, we can find an open neighborhood $V_z$ of $z$ so that for each $z' \in V_z$ there exists such a  curve $D_{z'}$ so that $D_{z'}\cap Z=\emptyset$ and $\partial D_{z'} \subset (U_k\setminus V_k) \cap X \subset U_k \cap X \setminus Z^{\epsilon}$. Furthermore, $\inf_{z'\in V_z}d(D_{z'}, Z)>0$ and $\inf_{z'}d(\partial D_{z'}, \partial (X\setminus Z^\epsilon))>0$.  As a result, we get an open cover $\cup_{z\in \overline{X\setminus Z^{\frac{\epsilon}{2}}}} V_z$ of $X\setminus Z^{\epsilon}$.  Since $\overline{X\setminus Z^{\frac{\epsilon}{2}}}$ is compact, we can find a finite subcover $\cup_{z_i} V_{z_i}$. Let $C(\epsilon)=\min_{i}\{\inf_{z\in V_{z_i}} d(D_{z}, Z),\inf_{z\in V_{z_i}}d(\partial D_{z}, \partial(X\setminus Z^{\epsilon}))\}$. This finishes the proof.  
\end{proof}

\subsubsection{Nontrivial limits}

\begin{prop}\label{nontrivial limit}
For any fixed $0<\epsilon<<1$, $s_i$ converges to a nontrivial holomorphic section $s_\infty \in H^0(X\setminus Z, \Ecal_\infty)$.
\end{prop}

\begin{proof}
We consider any $0<\epsilon<<1$ so that $X^\epsilon=
X\setminus Z^{\epsilon}$ satisfies the conditions needed in Lemma \ref{Curve property}. By induction, it suffices to show that 
$$
\|s_i\|_{L^\infty(X^{\frac{\epsilon}{2}})} \leq C:=C(\epsilon, \|s_i\|_{L^2(X^{\epsilon})})
$$
which implies the strong convergence of $s_i$ over $X^{\epsilon}$ for any $0<\epsilon<<1$. More precisely, given the estimates, we first have strong convergence over $X^{\epsilon}$, then over $X^\frac{\epsilon}{2}$, we have a uniform $L^\infty$ estimate, thus the inductive estimates 
$$
\|s_i\|_{L^\infty(X^{\frac{\epsilon}{4}})} \leq C:=C(\frac{\epsilon}{2}, \|s_i\|_{L^2(X^{\frac{\epsilon}{2}})})
$$
give strong convergence over $X^{\frac{\epsilon}{2}}$. Continuing this way, we obtain the convergence by passing to a subsequence. Now for any point $z\in X^{\frac{\epsilon}{2}}$, we want to prove 
$$
|s_i(z)| \leq C.
$$
Let $D$ be the holomorphic curve obtained in Lemma \ref{Curve property}. Let $t=s_i|_{D}$, we know 
$$
\Delta_{D} \log(|t|^2+1) \geq  -|F_{A_\infty}|_{D}| \geq -C_\epsilon.
$$
Here since the disk $D$ has a definite distance of $Z$ by assumption which might depend on $\epsilon$, we have $|F_{A_\infty}|D| \leq C_\epsilon$ for some $C_\epsilon$. Now near the point $z$, $X\cap U\subset U_k \subset \CBbb^n$, $\omega=i\partial \bar{\partial}\rho_k|_{X\cap U}$ for some smooth strongly plurisubharmonic function on $\overline{U_k}$. By definition, $\Delta_D \rho_k=1$ thus we have 
$$
\Delta_D (\log(|t|^2+1)+C_\epsilon \rho_k)\geq 0
$$
which by Maximum principle implies 
$$
\log(|t|^2+1)+C_\epsilon \rho_k \leq C_\epsilon \sup_{\partial D} \rho_k+\sup_{\partial D} \log(|t|^2+1) \leq C(\epsilon, \|s_i\|_{L^2(X^{\epsilon})}). 
$$
For the last inequality, the bound of the first term is trivial, and we only explain the second one. The interior estimate over $X^\epsilon$ for holomorphic sections implies $s_i$ is uniformly bounded within any precompact subset of $X^{\epsilon}$. In particular, it is uniformly bounded over $\partial D$ which lies in a fixed precompact subset of $X^\epsilon$,  
since $\partial D \subset X^\epsilon$ has a definite distance to $\partial X^\epsilon$.  This finishes the proof. 
\end{proof}

\section{Proof of Singular Donaldson-Uhlenbeck-Yau theorem}
We first prove the rank $1$ case. Recall we have a limiting Hermitian-Yang-Mills connection $A_\infty$ coming from the perturbed Hermitian-Yang-Mills connections and $A_\infty$ defines a holomorphic vector bundle $\Ecal_\infty$ over $X\setminus Z$ where $Z=X_s\cup \Sigma$ as before. 

By Proposition \ref{nontrivial limit}, we can conclude the existence of a nontrivial map $\Phi:  \Ecal_\infty \rightarrow \Ecal$ coming from the normalized identity map $\id: \Ecal \rightarrow \Ecal$. More precisely, to get the limit, we consider $\Hom(\Ecal, \Ecal)$ with the metrics $H_i^*\otimes H'$ where $H_i$ are the sequence of Hermitian-Yang-Mills metrics and $H'$ is any fixed smooth metric on $\Ecal$ away from $X_s$.  Now we apply Proposition \ref{nontrivial limit} to the sequence of holomorphic sections given by the identity map and get $\Phi: \Ecal_\infty \rightarrow \Ecal$.   

\begin{lem}\label{DUY for rank 1}
Given a rank $1$ reflexive sheaf $\Ecal$ over $X$, there exists an admissible Hermitian-Yang-Mills metric $H$ on $\Ecal$ with 
$$\mu(\Ecal)=\frac{\int_X c_1(H) \wedge \omega_1 \wedge \cdots \omega_{n-1}}{\rank \Ecal}.
$$  
\end{lem}

\begin{proof}
Since there exists a nontrivial map between $\Ecal$ and $\Ecal_\infty$ over $X^{reg}$, $\Ecal_\infty$ can be extended to be a reflexive rank $1$ sheaf over $X$. Indeed, by the Remmert-Stein extension theorem, $(\Ecal\otimes \Ecal_\infty^*)^{**}$ can be extended to be a rank $1$  reflexive sheaf from which the extension of $\Ecal_\infty$ follows trivially. It has to be an isomorphism since $\Ecal$ and $\Ecal_\infty$ have the same slope by Proposition \ref{Slope equality} and Lemma \ref{Nontrivial map must be isomorphism}. 
\end{proof}

In particular, for general rank, this gives 
\begin{cor}\label{Isomorphic determinant}
The induced admissible Hermitian-Yang-Mills connection on $\det(\Ecal_\infty)$ defines a sheaf isomorphic to $\det(\Ecal)$. Furthermore, $A_\infty$ computes the slope of $\det(\Ecal)$.
\end{cor}

The proof can be concluded by showing
\begin{prop}\label{Limit is semistable}
$\Phi:  \Ecal_\infty \rightarrow \Ecal$ is an isomorphism.
\end{prop}

\begin{proof}
Let $\Gcal$ denote the saturated subsheaf of $\Ecal$ given by the image of $\Phi$ in $\Ecal$. Following the argument in \cite[Section 7]{UhlenbeckYau:86}, by Siu's theorem (\cite{Siu:75}), $\Gcal$ can be extended to be a reflexive sheaf over $X$. Indeed, it suffices to prove it locally. We can assume $\mathcal{G}$ is a saturated subsheaf of $\Ocal^{\oplus N}$ over some open subset $U$ away from $\Sigma$. Then $\mathcal{G}\subset \Ocal^{\oplus N}$ defines a  map $f$ from the base to $Gr(N, N-\rank \Gcal)$ away from some codimension two subvariety $Z'$. By Siu's theorem (\cite[Page $441$]{Siu:75}), we know the graph $$\Gamma_f\subset (U\setminus \Sigma \cup X^s) \times Gr(N, N-\rank\Gcal)$$
can be extended to be a subvariety of $U\times Gr(N, N-\rank\Gcal)$ by taking its closure $\overline{\Gamma_f}$. Let $\pi_1: \overline{\Gamma_f} \rightarrow U$ and $\pi_2: \overline{\Gamma_f} \rightarrow Gr(N, N-\rank \Gcal)$ denote the natural projections. Let $\widetilde{\Fcal}$ denote the  tautological locally free sheaf over $Gr(N, N-\rank \Gcal)$. Then $((\pi_1)_*\pi_2^*\widetilde{\Fcal})^{**}$ gives an extension for $\Gcal$. Assume $\rank\Gcal <\rank \Ecal$. Since $\Ecal$ is stable, $\mu(\Gcal)<\mu(\Ecal)$. By considering 
$$
\Fcal:=\Hom(\wedge^{\rank \Gcal}\Ecal_\infty,  \det \Gcal) 
$$
 we can reduce to the following setting by Corollary \ref{DUY for rank 1}: a sheaf $\Fcal$ defined over $X^{reg}$ admits an admissible Hermitian-Yang-Mills metric with negative Einstein constant and a nonzero global section $s:=\wedge^{\rank \Gcal} \Phi$. By Corollary \ref{L infinity bound}, we know $s\in L^\infty$. Then we can use 
$$
\Delta |s|^2 =2 |\nabla s|^2-(\mu(\Gcal)-\mu(\Ecal)) |s|^2.
$$
By doing a cut-off argument, we know $\int_X \Delta |s|^2 =0$. This is a contradiction since the integral of the right hand side is strictly positive. In particular, $\Phi$ has full rank at a generic point. By Corollary \ref{Isomorphic determinant}, $\det\Phi$ is nowhere vanishing, thus $\Phi$ is an isomorphism.
\end{proof}

In particular, this gives 
\begin{cor}\label{Empty bubbling set}
$\Sigma=\emptyset$.
\end{cor}

\begin{proof}
Indeed, by Proposition \ref{Limit is semistable}, we know away from $X^s\cup \text{Sing}(\Ecal)$, the sequence of Hermitian-Yang-Mills metrics $H_{i}$ are actually locally uniformly equivalent if we normalize properly. In particular, fixing a background metric, we have a local uniform $C^0$ bound away from $X^s\cup \text{Sing}(\Ecal)$. Then by the elliptic estimates for Hermitian-Yang-Mills metric (\cite{BandoSiu:94}), we have local uniform higher order estimates for the sequence of metrics as well. In particular, $H_i$ converges smoothly away from $X^s\cup \text{Sing}(\Ecal)$. The conclusion follows.
\end{proof}

As a direct corollary,  we have the following
\begin{cor}\label{Cor3.5}
If a reflexive sheaf $\Fcal$ over $(X, \omega_1 \wedge\cdots \omega_{n-1})$ admits an admissible Hermitian-Yang-Mills metric $H$, it is polystable. 
\end{cor}

\begin{proof}
We first show that $H$ computes the slope of $\Fcal$. Indeed, by the singular Donaldson-Uhlenbeck-Yau theorem for the line bundle case, there exists a Hermitian-Yang-Mills metric $H_0$ on $\det(\Fcal)$ which computes the slope of $\det(\Fcal)$. Now  by Corollary \ref{HE uniqueness}, this defines the same Hermitian-Yang-Mills connection as $H$, in particular $H$ computes the slope of $\Fcal$. Let $\Fcal_1\subset \Fcal$ be a stable subsheaf of $\Fcal$ with $\mu(\Fcal_1)\geq \mu(\Fcal)$. By the Singular Donaldson-Uhlenbeck-Yau theorem, there exists an admissible Hermitian-Yang-Mills metric $H_1$ on $\Fcal_1$. Similar as the proof of Corollary \ref{HE uniqueness}, we can conclude that the map $\Fcal_1  \rightarrow \Fcal$ is parallel with respect to the connection defined by $H_1$ and $H$ on $\Hom(\Fcal_1, \Fcal)$ and $\mu(\Fcal_1)= \mu(\Fcal)$.  In particular, we know 
$$\Fcal=\Fcal_1 \oplus (\Fcal_1/\Fcal)^{**}.
$$ The conclusion now follows from induction.  
\end{proof}

In particular, Corollary \ref{Tensor being polystable} follows directly from this.

\section{Bogomolov-Gieseker inequality}

\subsection{General version}\label{Characterization of the Bogomolov-Gieseker inequality}
We separate the discussions into two cases. 
\subsubsection{Locally free resolutions}
Fix any K\"ahler resolution $p: \hat{X} \rightarrow X$ so that $\hat{\Ecal}=(p^*\Ecal)^{**}$ is locally free. Then
$$
\begin{aligned}
&\int_{\hat X} (2r c_2(\hat \Ecal)-(r-1)c_1(\hat \Ecal)^2) \wedge p^*\omega_{i_1} \wedge \cdots p^*\omega_{i_{n-2}}\\
=&\lim_i \int_{\hat X} (2r c_2(\hat \Ecal)-(r-1)c_1(\hat \Ecal)^2)\wedge (p^*\omega_{i_1}+i^{-1} \theta) \wedge \cdots (p^*\omega_{i_{n-2}}+i^{-1} \theta)\\
=& \lim_i \int_{\hat X} (2r c_2(A_i)-(r-1)c_1(A_i)^2)\wedge (p^*\omega_{i_1}+i^{-1} \theta) \wedge \cdots (p^*\omega_{i_{n-2}}+i^{-1} \theta)\\
\geq &\int_X (2r c_2(A_\infty)-(r-1)c_1(A_\infty)^2)\wedge \omega_{i_1} \wedge \cdots
 \omega_{i_{n-2}}.
\end{aligned}
$$
The first equality is trivial while the second follows from the Chern-Weil theory in the smooth case. For the last inequality, it follows from the  curvature of the induced connection on $\Hom(\hat\Ecal,\hat\Ecal)$ being primitive, thus 
$$(2r c_2(A_i)-(r-1)c_1(A_i)^2) \wedge (p^*\omega_{i_1}+i^{-1} \theta) \wedge \cdots (p^*\omega_{i_{n-2}}+i^{-1} \theta)$$ 
defines a sequence of Radon measures on $\hat X$ (see Equation \ref{Hodge Riemann property}). Furthermore, it has uniformly bounded mass. Then the inequality follows from Fatou's lemma. Since the limiting Hermitian-Yang-Mills connection is independent of the resolutions by Corollary \ref{HE uniqueness}, the conclusion follows.  About the case of equality, it follows from exactly the same argument as  \cite[Corollary 1.3]{ChenWentworth:21c}.
\subsubsection{General case} Fix a general K\"ahler resolutions $p:\hat{X} \rightarrow X$. Our main theorem implies there exists a family of admissible Hermitian-Yang-Mills metrics on $(p^*\Ecal)^{**}$ on $\hat X$ with respect to the perturbed metrics. Since over $\hat{X}$ the quantities 
$$(2r c_2(A_i)-(r-1)c_1(A_i)^2$$
given by the admissible Hermitian-Yang-Mills connections defines a sequence of closed currents (see \cite[Proposition $46$]{ChenWentworth:21a}), we know the quantities in the Bogomolov-Gieseker inequality can still be computed by the admissible Hermitian-Yang-Mills connections on $(p^*\Ecal)^{**}$ over $\hat{X}.$ Furthermore, exactly the same argument as the smooth case implies the family of admissible Hermitian-Yang-Mills connections converges to the admissible Hermitian-Yang-Mills connections on $\Ecal$ we obtained. So the same argument as above gives the statement about general resolutions.

\subsection{Various simplifications} 
Now we discuss various cases where the Bogomolov-Gieseker inequality could be simplified and one can extract more from the gauge theoretical picture when shrinking the exceptional divisor. 

\subsubsection{Over a normal surface} We first give a proof of Corollary \ref{Cor1.6}. The equality 
$$
\Delta(\Ecal) 
=(2rc_2(\Ecal)-(r-1)c_1(\Ecal)^2)\cap X
$$
follows from Proposition \ref{Pull-back minimizes}. Corollary \ref{Minimal resolution} follows from Corollary \ref{Bogomolov-Gieseker inequality} and Proposition \ref{Prop2.12}.

\subsubsection{When $X$ is smooth in codimension two} The equality 
$$
\begin{aligned}
&(2rc_2(\Ecal)-(r-1)c_1(\Ecal)^2).[\omega_{i_1}] \cdots  [\omega_{i_{n-2}}]\\
=&(2rc_2(\hat\Ecal)-(r-1)c_1(\hat\Ecal)^2).[p^*\omega_{i_1}] \cdots  [p^*\omega_{i_{n-2}}]\\
\end{aligned}
$$
follows from Proposition \ref{Prop2.13}. It remains to show that the Chern-Weil formula holds when $X$ has isolated singularities, i.e. 
$$
\begin{aligned}
&(2rc_2(\hat\Ecal)-(r-1)c_1(\hat\Ecal)^2).[p^*\omega_{i_1}] \cdots  [p^*\omega_{i_{n-2}}]\\
=& \int_{X} (2rc_2(H)-(r-1)c_1(H)^2)\wedge \omega_{i_1}\wedge \cdots \omega_{i_{n-2}}
\end{aligned}
$$
for any resolution $p:\hat{X}\rightarrow X$. For this, by definition, we can choose an embedding $U\rightarrow \CBbb^N$ near $x$ so that $\omega_i=\sqrt{-1}\partial\bar{\partial}(\rho_i|_U)$ for any $i$ where $\rho_i$ is a smooth function near $x$ in $\CBbb^n$. In particular, 
$$
\omega_{i_1}\wedge \cdots \omega_{i_{n-2}}|_U=d(\psi|_U)
$$
for some $\psi$ smooth on $\CBbb^N$, thus by cutting-off near the isolated singularities, we can globally write 
$$
\omega_{i_1}\wedge \cdots \omega_{i_{n-2}}=d\psi_0 +\Omega_0
$$
where $\psi_0$ is a restriction of form from $\CBbb^n$  that is smooth supported near $x$. By \cite[Lemma $2.5$]{ChenWentworth:21c}, we know 
$$
\int_{X} (2rc_2(H)-(r-1)c_1(H)^2)\wedge \omega_{i_1}\wedge \cdots \omega_{i_{n-2}}=\int_X \Omega_0.
$$
The conclusion follows.

\section{Bogomolov-Gieseker inequality for semistable sheaves} 
Now we go back to the general set-up and assume $\Ecal$ is a semistable torsion free sheaf over a normal variety $(X,\omega_1\wedge \cdots \omega_{n-1})$ endowed with $(n-1)$ K\"ahler metrics. The goal is to prove Theorem \ref{Bogomolov-Gieseker inequality}. For this, we first prove the inequality for semistable sheaves

\begin{prop}[Gieseker-Bogomolov inequality for semistable sheaves]\label{Proposition 5.1}
Given any semistable reflexive sheaf $\Ecal$ over $(X, \omega_1\wedge \cdots \omega_{n-1})$, the Gieseker-Bogomolov inequality holds i.e. 
$$
\Delta(\Ecal)[\omega_1].\cdots [\omega_{n-1}] \geq 0 
$$
\end{prop}

\begin{proof}
Fix $\hat{\Ecal}\in \Ebold_p$ for some resolution $p: \hat{X} \rightarrow X$. Let 
$$
0\subset \Ecal_1\subset \cdots \Ecal_m=\Ecal
$$
be a Jordan-H\"older filtration for $\Ecal$. Then this naturally gives a filtration for 
$$
0 \subset \hat{\Ecal}_1 \subset \cdots \hat{\Ecal}_m=\hat{\Ecal}
$$
so that $\hat{\Ecal}_k/\hat{\Ecal}_{k-1}$ is torsion free for any $k$ and $\hat{\Ecal}_k|_{X^{reg}}\cong \Ecal_k|_{X^{reg}}$.  By Proposition \ref{Openess of stability}, we know for any $k$, $\hat{\Ecal}_k/\hat{\Ecal}_{k-1}$ is stable with respect to $[p^*\omega_1+i^{-1}\theta ] \wedge \cdots [p^*\omega_{n-1}+i^{-1}\theta]$. Thus 
$$
\Delta(\hat{\Ecal}_k/\hat{\Ecal}_{k-1}).[p^*\omega_1+i^{-1}\theta ] \wedge \cdots [p^*\omega_{n-1}+i^{-1}\theta]\geq 0.
$$
Denote the Chern classes of $ \hat{\Ecal}_k/\hat{\Ecal}_{k-1}$ (resp. $\hat{\Ecal}$) by $\gamma_k$ (resp. $\gamma$), and the  rank of $ \hat{\Ecal}_k/\hat{\Ecal}_{k-1}$ (resp. $\hat{\Ecal}$) by $r_k$ (resp. $r$). Then $r=\sum_k \gamma_k$, and $\gamma=\sum_{k}r_k$. By definition, we have 
$$
\sum_k \frac{\Delta(\hat{\Ecal}_k/\hat{\Ecal}_{k-1})}{r_k}-\frac{\Delta(\hat{\Ecal})}{r}= \sum_{s<t}\frac{1}{r} r_sr_t (\frac{\gamma_s}{r_s}-\frac{\gamma_t}{r_t})^2.
$$
By assumption, we know 
$$
(\frac{\gamma_s}{r_s}-\frac{\gamma_t}{r_t})[p^*\omega_1] \cdots [p^*\omega_{n-1}]=0.
$$
To finish the proof, it suffices to prove that 
\begin{equation}
(\frac{\gamma_s}{r_s}-\frac{\gamma_t}{r_t})^2 [p^*\omega_{i_1}] \cdots [p^*\omega_{i_{n-1}}]\leq 0
\end{equation}
In the smooth case, this follows from \cite{DinhNguyen:06} which generalize the classical Hodge-Riemann theorem over K\"ahler manifolds to the case of multiple K\"ahler metrics (\cite{Timorin:98}). Let $c_{st}^i$ so that 
$$
(\frac{\gamma_s}{r_s}-\frac{\gamma_t}{r_t}-c_{st}^i[p^*\omega_{i_{n-1}}])[p^*\omega_{i_1}+i^{-1}\theta] \cdots [p^*\omega_{i_{n-2}}+i^{-1}\theta]=0
$$
which implies 
$$
(\frac{\gamma_s}{r_s}-\frac{\gamma_t}{r_t}-c_{st}^i[p^*\omega_{i_n}])(\frac{\gamma_s}{r_s}-\frac{\gamma_t}{r_t}-c_{st}^i[p^*\omega_{i_n}])[p^*\omega_{i_1}+i^{-1}\theta] \cdots [p^*\omega_{i_{n-2}}+i^{-1}\theta]\leq 0.
$$
Since $c_{st}^i \rightarrow 0$ as $i\rightarrow \infty$, by taking the limit of the equation above, we have
$$
(\frac{\gamma_s}{r_s}-\frac{\gamma_t}{r_t})(\frac{\gamma_s}{r_s}-\frac{\gamma_t}{r_t})[p^*\omega_{i_1}] \cdots [p^*\omega_{i_{n-2}}]\leq 0.
$$ 
The conclusion follows.
\end{proof}

Assume now
$$\Delta(\hat{\Ecal}).[p^*\omega_{i_1}] \cdots [p^*\omega_{i_{n-2}}]=0$$ 
for some $\hat{\Ecal}\in \Ebold_p$ and some resolution $p: \hat{X} \rightarrow X$. To finish the proof of Theorem \ref{Theorem1.16}, it remains to show 
\begin{prop}
$\Ecal$ admits a filtration 
$$
0\subset \Ecal_1 \subset \cdots \Ecal_{m}=\Ecal
$$
so that $\Ecal_{i}/\Ecal_{i-1}$ is torsion free and $(\Ecal_{i}/\Ecal_{i-1})|_{X^{reg}}$ is projectively flat. 
\end{prop}

\begin{proof}
Let $\Ecal_1$ (resp. $\hat{\Ecal}_1$) be a stable subsheaf of $\Ecal$ (resp. $\hat{\Ecal}$) having the same slope as $\Ecal$ (resp. $\hat{\Ecal}$) so that $\Ecal/\Ecal_1$ (resp. $\hat{\Ecal}/\hat{\Ecal}_1$) is torsion free stable with the same slope as $\Ecal$ (resp. $\hat{\Ecal})$, and 
$$
\hat{\Ecal}_1|_{X^{reg}}\cong \Ecal_1|_{X^{reg}}.
$$ 
By induction, we need to show 
\begin{enumerate}
\item $\hat{\Ecal}_1$ is projectively flat away from the exceptional divisor.

\item $\hat{\Ecal}/\hat{\Ecal_1}$ is reflexive away from the exceptional divisor and 
$$
\Delta((\hat{\Ecal}/\hat{\Ecal}_1)^{**}).[p^*\omega_{i_1}] \cdots [p^*\omega_{i_{n-2}}]=0.
$$
\end{enumerate}
We first prove $(1)$. Since $\hat{\Ecal}_1$ and $\hat{\Ecal}/\hat{\Ecal}_1$ are stable, by Corollary \ref{Bogomolov-Gieseker inequality}, we know
$$
\Delta(\hat{\Ecal}_1).[p^*\omega_{i_1}] \cdots [p^*\omega_{i_{n-2}}]\geq 0.
$$
and 
$$
\Delta((\hat{\Ecal}/\hat{\Ecal}_1)^{**}).[p^*\omega_{i_1}] \cdots [p^*\omega_{i_{n-2}}]\geq 0
$$
By Corollary \ref{Bogomolov-Gieseker inequality}, it suffices to show 
$$
\Delta(\hat{\Ecal}_1).[p^*\omega_{i_1}] \cdots [p^*\omega_{i_{n-2}}]=0.
$$
As the proof of Proposition \ref{Proposition 5.1}, we know 
\begin{equation}
(\frac{\Delta(\hat{\Ecal}_1)}{r_1}+\frac{\Delta(\hat{\Ecal}/\hat{\Ecal}_1)}{r_1}-\frac{\Delta(\hat{\Ecal})}{r}) [p^*\omega_{i_1}] \cdots [p^*\omega_{i_{n-1}}]\leq 0
\end{equation}
which combined with the inequality above implies 
$$\Delta(\hat{\Ecal}_1)=0$$ 
and 
$$
\frac{\Delta(\hat{\Ecal}/\hat{\Ecal}_1)}{r_1}[p^*\omega_{i_1}] \cdots [p^*\omega_{i_{n-1}}]= 0
$$
if we can show 
\begin{equation}\label{Equation5.3}
\Delta(\hat{\Ecal}/\hat{\Ecal}_1)[p^*\omega_{i_1}] \cdots [p^*\omega_{i_{n-2}}] \geq \Delta((\hat{\Ecal}/\hat{\Ecal}_1)^{**}).[p^*\omega_{i_1}] \cdots [p^*\omega_{i_{n-2}}].
\end{equation}
Assume this, by Corollary \ref{Bogomolov-Gieseker inequality}, we know $\hat{\Ecal}_1$ is projectively flat over $X^{reg}$. Now we prove Equation \ref{Equation5.3}. Indeed, we have 
$$
0\rightarrow\hat{\Ecal}/\hat{\Ecal}_1 \rightarrow (\hat{\Ecal}/\hat{\Ecal}_1)^{**} \rightarrow \tau \rightarrow 0
$$
where $supp(\tau)$ has codimension at least two. Let $\Sigma_k$ denotes the irreducible pure codimension two components of $supp(\tau)$, and to each $\Sigma_k$, one can associate it with an analytic multiplicity $m_k=h^0(\tau|_\Delta, \Sigma_k)$ where $\Delta$ is a transverse slice of $\Sigma_k$ at a generic point. Then 
$$\ch_2(\tau)=\PD(\sum_k m_k\Sigma_k)$$ 
by \cite[Proposition $3.1$]{SibleyWentworth:15}. We then know from definition that 
$$
\begin{aligned}
&\Delta(\hat{\Ecal}/\hat{\Ecal}_1).[p^*\omega_{i_1}] \cdots [p^*\omega_{i_{n-2}}] -\Delta((\hat{\Ecal}/\hat{\Ecal}_1)^{**}).[p^*\omega_{i_1}] \cdots [p^*\omega_{i_{n-2}}]  \\
=& \sum_k m_k\int_{\Sigma_k} p^*\omega_{i_1} \wedge \cdots \omega_{i_{n-2}} \\
\geq & 0
\end{aligned}
$$
where the equality holds if and only if $\Sigma_k\cap X^{reg}=\emptyset$ for any $k$.

For (2), from the discussion above, we know $\text{supp}(\tau)$ has codimension at least three over $X^{reg}$. Furthermore, we have 
$$\Delta((\hat{\Ecal}/\hat{\Ecal}_1)^{**}).[p^*\omega_{i_1}] \cdots [p^*\omega_{i_{n-2}}]=0.$$
By induction, $(\hat{\Ecal}/\hat{\Ecal}_1)^{**}$ is locally free over $X^{reg}$.  It remains to show $\hat{\Ecal}/\hat{\Ecal}_1$ is reflexive over $X^{reg}$.
This follows from the argument in \cite[Proposition $2.3$]{SibleyWentworth:15}. The key point is that over $X^{reg}$, the short exact sequence 
$$
0\rightarrow \hat{\Ecal}_1 \rightarrow \hat{\Ecal} \rightarrow \hat{\Ecal}/\hat{\Ecal}_1  \rightarrow 0
$$
locally splits away from a codimension three set over $X^{reg}$ since $\hat{\Ecal}_1$ and $(\hat{\Ecal}/\hat{\Ecal}_1)^{**}$ are locally free over $X^{reg}$. Thus this will force $\hat{\Ecal}/\hat{\Ecal}_1$ to be reflexive over $X^{reg}$, i.e. locally free over $X^{reg}$ in this case. 
\end{proof}

\bigskip

\noindent {\bf Data Availability Statement}. Data sharing not applicable to this article as no datasets were generated or analyzed during the current study.

\medskip

\noindent{\bf Conflict of Interest Statement}.
On behalf of all authors, the corresponding author states that there is no conflict of interest.

\bibliography{papers}

\end{document}